\documentclass[12pt]{article}

\usepackage[cp1251]{inputenc}
\usepackage[english]{babel}
\usepackage[T1, T2A]{fontenc}

\usepackage{amssymb}
\usepackage{amsmath}
\usepackage{amsthm} 
\usepackage{mathtext}
\usepackage{amsfonts}

\renewcommand{\leq}{\leqslant}
\renewcommand{\geq}{\geqslant}

\newcommand{\const}{\operatorname{const}}

\newcommand{\lda}{\lambda}

\newcommand{\Cspec}{\mathfrak{C}}

\newtheorem{rustheorem}{Theorem}
\newtheorem{ruslemma}{Lemma}
\newtheorem{rusproposition}{Proposition}

\theoremstyle{definition}

\newtheorem{rusremark}{Remark}
\newtheorem{rusexample}{Example}

\newenvironment{rusbibliography}{\vspace{-0.5cm}\begin{thebibliography}}{\end{thebibliography}}

\begin{document} 
\title{On spectral asymptotics of the tensor product of operators with almost regular marginal asymptotics}
\author{N.~V.~Rastegaev \\ \small{Chebyshev Laboratory, St. Petersburg State 
University,} \\
\small{14th Line 29B, Vasilyevsky Island, 199178, St. Petersburg, Russia} 
\\ \small{rastmusician@gmail.com}}
\renewcommand{\today}{}
\maketitle

\abstract{
Spectral asymptotics of a tensor product of compact operators in Hilbert space
 with known marginal asymptotics is studied. Methods of A. Karol', A. Nazarov
 and Ya. Nikitin (Trans. AMS, 2008) are generalized for operators with almost
 regular marginal asymptotics. In many (but not all) cases it is shown, that
 tensor product has almost regular asymptotics as well. Obtained results are
 then applied to the theory of small ball probabilities of Gaussian random fields.
 }

\section{Introduction}
We consider compact nonnegative self-adjoint operators $\mathcal{T} = \mathcal{T}^*\geqslant 0$ in a Hilbert space $\mathcal H$ and $\widetilde{\mathcal{T}}$ in a Hilbert space $\widetilde{\mathcal H}$.
We denote by $\lda_n = \lda_n(\mathcal{T})$ 
the eigenvalues of the operator $\mathcal{T}$ arranged in a nondecreasing order and repeated according to their multiplicity. We also consider their counting function
$$
\mathcal{N}(t) = \mathcal{N}(t, \mathcal{T}) = \#\{n: \lda_n(\mathcal{T}) > t\}.
$$
Similarly we define $\widetilde{\lda}_n$ and $\widetilde{\mathcal{N}}(t)$ for $\widetilde{\mathcal{T}}$.

Having known asymptotics for $\mathcal{N}(t, \mathcal{T})$ and $\mathcal{N}(t, \widetilde{\mathcal{T}})$ as $t\to 0$, we aim to determine the asymptotics for $\mathcal{N}(t, \mathcal{T}\otimes\widetilde{\mathcal{T}})$. Obtained results are easily generalized to the case of a tensor product of multiple operators.

Known applications of such results could be found in problems concerning asymptotics of random values and vectors quantization (see e.g. \cite{GrLuPa,LuPa}), average complexity of linear problems, i.e. problems of approximation of a continuous linear operator (see e.g. \cite{PaWa}), and also in the developing theory of small deviations of random processes in $L_2$-norm (see e.g. \cite{NazNikKar,NazKar}).

Abstract methods of spectral asymptotics analysis for tensor products, generalized in this paper, were developed in \cite{NazNikKar} and \cite{NazKar}.
In \cite{NazNikKar} the case is considered, in which the eigenvalues of the operators-multipliers have the so-called \textit{regular} asymptotic behavior:
$$
\lda_n \sim \dfrac{\psi(n)}{n^{p}}, \quad n\to\infty,
$$
where $p > 1$, and $\psi$ is a \textit{slowly varying} function (SVF).
In the paper \cite{NazKar} similar approach is used for the case, when the eigenvalue counting function has the asymptotics of a slowly varying function.

In this paper we consider compact operators with almost regular asymptotics
\begin{equation}\label{asymplda}
\lda_n(\mathcal{T}) \sim \dfrac{\psi(n)\cdot \mathfrak{s}(\ln(n))}{n^p}, \quad n\to+\infty,
\end{equation}
where $p > 1$, $\psi$ is a SVF, and $\mathfrak{s}$ is a continious periodic function. As an example of such an operator one might consider a Green integral operator with a singular arithmetically self-similar weight measure (see \cite{KL,SV,Naz}).

For the asymptotics \eqref{asymplda} the following fact holds, which is similar to Lemma~3.1 from \cite{NazNikKar}, so we will provide it without proof.

\begin{rusproposition}\label{proposEIGEN}
For any $p>0$ spectral asymptotics \eqref{asymplda} for the operator $\mathcal{T}$ is equivalent to the asymptotics
\begin{equation}\label{Tasymp}
\mathcal{N}(t, \mathcal{T}) \sim \mathcal{N}_{as}(t) := 
\dfrac{\varphi(1/t)\cdot s(\ln (1/t))}{t^{1/p}}, \qquad t\to+0,
\end{equation}
where $\varphi$ is a SVF, $s$ is a periodic function {\rm(}period $T$ of the function $s$ corresponds to the period $T/p$ of the function $\mathfrak{s}${\rm)}.
Moreover, the convergence of the integral $\int_1^\infty \varphi(\tau)\frac{d\tau}{\tau}$ is equivalent to the convergence of the sum $\sum_n\lda_n^{1/p}(\mathcal{T})$.
\end{rusproposition}

\noindent Application of the obtained results is demonstrated for an example on the problem of $L_2$-small deviations of Gaussian random fields.

The study of the small deviation problem was initiated in \cite{Sytaya} and continued by many other scholars. The
history of the problem and the summary of main results are the subjects of two 
reviews \cite{Lifsh} and \cite{LiShao}. Links to recent results in the field of small deviations of random processes could be found on the web-site \cite{Site}.

The study of small deviations of Gaussian fields of a tensor product type was initiated in the classical paper \cite{Csaki}, where the logarithmic asymptotics of $L_2$-small deviations was obtained for the Brownian sheet
$$
\mathbb{W}_d(x_1, \dots, x_d) = W_1(x_1) \otimes W_2(x_2) \otimes \dots \otimes
W_d(x_d)
$$
on the unit cube (here $W_k$ are independent Wiener processes).
This result was later generalized in \cite{Li} to some other marginal processes. In \cite{NazNikKar} and \cite{NazKar} the results on small deviations of wide classes of Gaussian fields of tensor product type were obtained as a corollary of the results on spectral asymptotics of the corresponding operators.

This paper has the following structure. We give the necessary information about slowly varying functions in \S 2. In \S 3 we establish some auxiliary facts related to the asymptotics of the convolutions of an almost Mellin type. 

Spectral asymptotics of the tensor products of operators with marginal asymptotics of the form \eqref{Tasymp} are the subject of \S 4. The main results are that we obtain the main term of the spectral asymptotics of the tensor product for all possible combinations of the parameters, imposing only slight technical restrictions in some cases. The results are separated into several cases depending on the relations between the parameters of the spectral asymptotics of the operators $\mathcal{T}$ and $\widetilde{\mathcal{T}}$:
\begin{enumerate}
\item $\widetilde p > p$.
\item $\widetilde p = p$.
\begin{enumerate}
\item $\int\limits_1^\infty \varphi(\sigma) \dfrac{d\sigma}{\sigma} = \int\limits_1^\infty \widetilde\varphi(\sigma) \dfrac{d\sigma}{\sigma} = \infty$.
\begin{enumerate}
\item Functions $s$ and $\widetilde s$ have a common period ($T=\widetilde T$).
\item Periods $T$ and $\widetilde T$ of the functions $s$ and $\widetilde s$ are incommensurable.
\end{enumerate}
\item $\int\limits_1^\infty \varphi(\sigma) \dfrac{d\sigma}{\sigma} < \infty,\quad \int\limits_1^\infty \widetilde\varphi(\sigma) \dfrac{d\sigma}{\sigma} = \infty$.
\item $\int\limits_1^\infty \varphi(\sigma) \dfrac{d\sigma}{\sigma} <\infty,\quad \int\limits_1^\infty \widetilde\varphi(\sigma) \dfrac{d\sigma}{\sigma} < \infty$.
\end{enumerate}
\end{enumerate}
In the cases 1, 2.1.1 the asymptotics of the tensor product is almost regular, in the case 2.1.2 it is regular. In the cases 2.2 and 2.3 we obtain an asymptotics of a more complex form.

In \S 5 we connect the almost regular spectral asymptotics with the logarithmic asymptotics of small deviations of Gaussian random fields.

Different constants with values, that are not essential for this work, are denoted~$C$. Dependence of this constants on parameters is noted in parentheses.

\section{Auxiliary facts about slowly varying functions}
We recall, that a positive function $\varphi(\tau)$, $\tau > 0$, is called {\it slowly varying} (at infinity), if for any constant $c>0$
\begin{equation}\label{SVFrel}
\varphi(c\tau)/\varphi(\tau)\to 1, \quad \text{ as } \tau\to +\infty.
\end{equation} 
The following simple properties of SVFs are known (see e.g. \cite{Seneta} for proofs).
\begin{rusproposition}\label{propSVF1}
Let $\varphi$ be an SVF. Then the following properties hold:
\begin{enumerate}
\item The convergence in \eqref{SVFrel} is uniform for $c\!\in\![a,b]$ for any ${0\!<\!a\!<\!b\!<\!+\infty}$.
\item Function $\tau\mapsto\tau^p\varphi(\tau)$, $p\neq 0$, is monotonous for large values of $\tau$.
\item There exists an equivalent SVF $\psi\in C^2(\mathbb{R})$ {\rm(}i.e.  $\varphi(\tau)/\psi(\tau)\to 1$ as $\tau\to\infty${\rm)}, such that
$$
\tau\cdot (\ln(\psi))'(\tau)\to 0, \quad \tau^2\cdot (\ln(\psi))''(\tau)\to 0, 
\qquad \tau\to\infty.
$$
\item If $\int_1^\infty \varphi(\tau)\frac{d\tau}{\tau}<\infty$, then $\varphi(\tau)\to 0$
as $\tau\to\infty$.
\end{enumerate}
\end{rusproposition}

\noindent Following \cite{NazNikKar}, we define the {\it Mellin convolution} of two SVFs $\varphi$ and $\psi$\,:
$$
(\varphi\ast\psi)(\tau) = \int\limits_1^\tau\varphi(\sigma)\psi(\tau/\sigma)\dfrac{d\sigma}{\sigma}
= h_{\varphi,\psi}(\tau)+h_{\psi, \varphi}(\tau),
$$
where
$$
h_{\varphi, \psi}(\tau) = \int\limits_1^{\sqrt{\tau}}\varphi(\sigma)\psi(\tau/\sigma)\dfrac{d\sigma}{\sigma}.
$$
\begin{rusproposition}[{\cite[Theorem 2.2]{NazNikKar}}]\label{propSVF2} The following properties hold:
\begin{enumerate}
\item If $\int_1^\infty \varphi(\tau)\frac{d\tau}{\tau}=\infty$, then 
$\psi(\tau) = o(h_{\varphi, \psi}(\tau))$ as $\tau\to\infty$.
\item If $\psi(\tau) = \psi_1(\tau)(1+o(1))$ as $\tau\to\infty$, then 
$$
h_{\varphi, \psi}(\tau) = h_{\varphi, \psi_1}(\tau)(1+o(1)), \quad\tau\to\infty.
$$
If also $\int_1^\infty\psi(\tau)\frac{d\tau}{\tau} = \infty$, then
$$
h_{\psi, \varphi}(\tau) = h_{\psi_1, \varphi}(\tau)(1+o(1)), \quad\tau\to\infty.
$$
\item $h_{\varphi, \psi}$ is a SVF.
\item Let $\int_1^\infty \varphi(\tau)\frac{d\tau}{\tau}<\infty$, and also
$$
	\int\limits_1^\infty \varphi(\sigma)m_\psi(\sigma)\frac{d\sigma}{\sigma}<\infty,
$$
where 
$$
m_\psi(\sigma) = \sup_{\tau>\sigma^2}\dfrac{\psi(\tau/\sigma)}{\psi(\tau)}.
$$
Then
\begin{equation}\label{SVFestim1}
h_{\varphi, \psi}(\tau) = \psi(\tau)\int\limits_1^\infty\varphi(\sigma)\dfrac{d\sigma}{\sigma}
\cdot (1+o(1)), \qquad \tau\to\infty.
\end{equation}
\end{enumerate}
\end{rusproposition}

\section{Preliminary facts about the asymptotics of almost Mellin convolutions}

In this section $\varphi$ and $\widetilde\varphi$ are SVFs, $s$ and $\widetilde s$ are continious, bounded, separated from zero functions with periods $T$ and $\widetilde T$ correspondingly that have the form
$$
s(\tau) = e^{-\tau/p}\varrho(\tau), \qquad \widetilde s(\tau) = e^{-\tau/p}\widetilde\varrho(\tau),
$$
where $p>0$, $\varrho$ and $\widetilde\varrho$ are monotonous. It means, in particular, that $s$ are $\widetilde s$ are bounded variation functions.

We define {\it almost Mellin convolution}
\begin{align*}
(\varphi s \ast \widetilde \varphi \widetilde s)(\tau)  & = 
\int\limits_{1}^{\tau}
\varphi\left(\frac{\tau}{\sigma}\right)\widetilde\varphi(\sigma) s\left(\ln\frac{\tau}{\sigma}\right)
\widetilde s(\ln\sigma)\dfrac{d(\widetilde{\varrho}(\ln \sigma))}{\widetilde{\varrho}(\ln \sigma)}  {}
\\
& {} = H[\varphi s, \widetilde \varphi \widetilde s](\tau) + H_1[\varphi s, \widetilde \varphi \widetilde s](\tau), \\
H[\varphi s, \widetilde \varphi \widetilde s](\tau) & = \int\limits_{1}^{\sqrt{\tau}}
\varphi\left(\frac{\tau}{\sigma}\right)\widetilde\varphi(\sigma) s\left(\ln\frac{\tau}{\sigma}\right)
\widetilde s(\ln\sigma)\dfrac{d(\widetilde{\varrho}(\ln \sigma))}{\widetilde{\varrho}(\ln \sigma)}, \\
H_1[\varphi s, \widetilde \varphi \widetilde s](\tau) & = \int\limits_{\sqrt{\tau}}^{\tau}
\varphi\left(\frac{\tau}{\sigma}\right)\widetilde\varphi(\sigma) s\left(\ln\frac{\tau}{\sigma}\right)
\widetilde s(\ln\sigma)\dfrac{d(\widetilde{\varrho}(\ln \sigma))}{\widetilde{\varrho}(\ln \sigma)}.
\end{align*}
The integral here should be interpreted as a Lebesgue--Stieltjes integral.

\begin{ruslemma}\label{MellinConv1}
\begin{align*}
(\varphi s \ast \widetilde \varphi \widetilde s)(\tau) &\asymp (\varphi\ast\widetilde\varphi)(\tau), \quad \tau\to\infty,
\\
H[\varphi s, \widetilde \varphi \widetilde s](\tau) &\asymp h_{\widetilde\varphi, \varphi}(\tau), \quad \tau\to\infty,
\\
H_1[\varphi s, \widetilde \varphi \widetilde s](\tau) &\asymp h_{\varphi, \widetilde\varphi}(\tau), \quad \tau\to\infty.
\end{align*}
\end{ruslemma}
\begin{proof}
We will prove the upper estimate for the first relation, the rest of the estimates could be obtained similarly. Let us introduce an operator
\begin{equation}\label{defF}
F_{\sigma}[\varphi](\xi) = \varphi(e^{j\widetilde T}\sigma)\ \text{ for }\  \xi\in[e^{j\widetilde T}, e^{(j+1)\widetilde T}),
\end{equation}
that transforms a given function $\varphi$ into a step-wise function.

\noindent Note, that
\begin{equation}\label{asympF}
F_{\sigma}[\varphi](\tau) = \varphi(\tau)(1+o(1)), \quad \tau\to\infty
\end{equation}
uniformly over
$\sigma\in[1, e^{\widetilde T}]$.
Let $k\in\mathbb{N}$ be such a number, that $e^{(k-1)\widetilde T} < \tau \leq e^{k\widetilde T}$. Then
$$
(\varphi s \ast \widetilde \varphi \widetilde s)(\tau) \leq C \int\limits_{1}^{e^{k\widetilde{T}}}
F_{e^{k\widetilde T}/\tau}[\varphi]\left(\frac{\tau}{\sigma}\right)F_1[\widetilde\varphi](\sigma) \dfrac{d(\widetilde{\varrho}(\ln \sigma))}{\widetilde{\varrho}(\ln \sigma)}.
$$
Note, that function $F_{e^{k\widetilde T}/\tau}[\varphi]\Big(\dfrac{\tau}{\sigma}\Big)F_1[\widetilde\varphi](\sigma)$ is constant with respect to $\sigma$ on every interval $(e^{j\widetilde T}, e^{(j+1)\widetilde T})$, $j=0,\ldots, k-1$. Measure $\dfrac{d(\widetilde{\varrho}(\ln \sigma))}{\widetilde{\varrho}(\ln \sigma)} \!= \!d \ln( \widetilde{s}(\ln \sigma) \sigma^{1/p} )$ is periodic with respect to $\ln\sigma$, which allows us to replace the integral with a sum. We obtain
\begin{equation*}
\begin{split}
(\varphi s \ast \widetilde \varphi \widetilde s)(\tau) & \leq C\int\limits_1^{e^{\widetilde T}}\dfrac{d(\widetilde{\varrho}(\ln \sigma))}{\widetilde{\varrho}(\ln \sigma)} \sum_{j=0}^{k-1} F_{e^{k\widetilde T}/\tau}[\varphi](\frac{\tau}{e^{j\widetilde T}})F_1[\widetilde\varphi](e^{j\widetilde T}) 
\\
&{}\leq C \int\limits_1^{e^{\widetilde T}} \dfrac{d\sigma}{\sigma} \sum_{j=0}^{k-1} F_{e^{k\widetilde T}/\tau}[\varphi](\frac{\tau}{e^{j\widetilde T}})F_1[\widetilde\varphi](e^{j\widetilde T})
\\
&{} = C \int\limits_{1}^{e^{k\widetilde{T}}}
F_{e^{k\widetilde T}/\tau}[\varphi]\left(\frac{\tau}{\sigma}\right)F_1[\widetilde\varphi](\sigma) \dfrac{d\sigma}{\sigma} \leq C \int\limits_{1}^{\tau}
\varphi\left(\frac{\tau}{\sigma}\right)\widetilde\varphi(\sigma) \dfrac{d\sigma}{\sigma}.\qedhere
\end{split}
\end{equation*}
\end{proof}

\noindent The proof of the following proposition is similar to Theorem 2.2 from \cite{NazNikKar}, so we omit it.

\begin{rusproposition}\label{SVFasymp}
Let $\widetilde\varphi(\tau) = \psi_1(\tau)(1+o(1))$ as $\tau\to\infty$. Then
\begin{align*}
H[\varphi s, \widetilde \varphi \widetilde s](\tau) =  H[\varphi s, \psi_1 \widetilde s](\tau)(1+o(1)).
\end{align*}
If, also, $\int\limits_1^\infty\widetilde\varphi(\tau)\dfrac{d\tau}{\tau} = \infty$, then
\begin{align*}
H_1[\varphi s, \widetilde \varphi \widetilde s](\tau) =  H_1[\varphi s, \psi_1 \widetilde s](\tau)(1+o(1)).
\end{align*}
\end{rusproposition}

\begin{ruslemma}\label{MellinConv2}
Let $\int\limits_1^\infty\varphi(\tau)\dfrac{d\tau}{\tau} = \infty$, $\int\limits_1^\infty\widetilde\varphi(\tau)\dfrac{d\tau}{\tau} = \infty$. Then
$$
H_1[\varphi s, \widetilde \varphi \widetilde s](\tau) = H[\widetilde \varphi \widetilde s, \varphi s](\tau)(1+o(1)), \quad \tau\to\infty,
$$
and almost Mellin convolution is asymptotically symmetric, i.e.
$$
(\varphi s \ast \widetilde \varphi \widetilde s)(\tau) =  (\widetilde \varphi \widetilde s \ast \varphi s )(\tau)(1+o(1)), \quad \tau\to\infty.
$$
\end{ruslemma}
\begin{proof}
The second relation follows from the first one immediately. In order to prove the first one we are going to transform the relation
$$
H_1[\varphi s, \widetilde \varphi \widetilde s](\tau) =
\tau^{-1/p}\int\limits_{\sqrt{\tau}}^\tau
\varphi\left(\frac{\tau}{\sigma}\right)\widetilde\varphi(\sigma) \rho\left(\ln\frac{\tau}{\sigma}\right)
d(\widetilde \rho(\ln\sigma)).
$$
Let us replace $\sigma$ with $\tau/\sigma$ and integrate by parts.
\begin{multline*}
H_1[\varphi s, \widetilde \varphi \widetilde s](\tau) =
-\tau^{-1/p}\int\limits_1^{\sqrt{\tau}}
\varphi(\sigma)\widetilde\varphi\left(\frac{\tau}{\sigma}\right) \rho(\ln\sigma)
d\left(\widetilde \rho\left(\ln\frac{\tau}{\sigma}\right)\right) 
\\
 = \tau^{-1/p}\int\limits_1^{\sqrt{\tau}}
\varphi(\sigma)\widetilde\varphi\left(\frac{\tau}{\sigma}\right)
\widetilde \rho\left(\ln\frac{\tau}{\sigma}\right) d(\rho(\ln\sigma)) +
\left. \varphi(\sigma)\widetilde\varphi\left(\frac{\tau}{\sigma}\right) s(\ln\sigma)
\widetilde s\left(\ln\frac{\tau}{\sigma}\right)\right|_1^{\sqrt{\tau}}
\\
+\int\limits_1^{\sqrt{\tau}}
\Big(
\dfrac{\sigma\varphi'(\sigma)}{\varphi(\sigma)}-
\dfrac{(\tau/\sigma)\widetilde\varphi'(\tau/\sigma)}{\widetilde\varphi(\tau/\sigma)}
\Big)
\varphi(\sigma)\widetilde\varphi\left(\frac{\tau}{\sigma}\right)
\widetilde s\left(\ln\frac{\tau}{\sigma}\right) s(\ln\sigma) \dfrac{d\sigma}{\sigma}.
\end{multline*}
The first term equals $H[\widetilde \varphi \widetilde s, \varphi s](\tau)$.
What remains is to show, that the second and the third terms satisfy the estimate
$o(H[\widetilde \varphi \widetilde s, \varphi s](\tau))$. Let's look at the second term.
\begin{align*}
&\left. \varphi(\sigma)\widetilde\varphi\left(\frac{\tau}{\sigma}\right) s(\ln\sigma)
\widetilde s\left(\ln\frac{\tau}{\sigma}\right)\right|_1^{\sqrt{\tau}} \\
&\qquad =
\varphi(\sqrt{\tau})\widetilde\varphi(\sqrt{\tau}) s(\ln\sqrt{\tau})
\widetilde s(\ln\sqrt{\tau})
- \varphi(1)\widetilde\varphi(\tau) s(0)
\widetilde s(\ln\tau).
\end{align*}
All periodic components are bounded.
$$
\widetilde\varphi(\tau) = o(h_{\varphi, \widetilde\varphi}(\tau)) = o(H[\widetilde \varphi \widetilde s, \varphi s](\tau)), \quad \tau\to\infty
$$
according to Proposition~\ref{propSVF2}, part 1, in view of Lemma~\ref{MellinConv1}. Thus, it is sufficient to estimate
{\allowdisplaybreaks
\begin{align*}
\varphi(\sqrt{\tau})\widetilde \varphi(\sqrt{\tau})& = 
\varphi(1)\widetilde\varphi(\tau) + \int\limits_1^{\sqrt{\tau}} 
\Big( \varphi(\sigma)\widetilde\varphi\left(\frac{\tau}{\sigma}\right) \Big)'_\sigma \,d\sigma
\\
& = \varphi(1)\widetilde\varphi(\tau) + \int\limits_1^{\sqrt{\tau}}\Big(
\dfrac{\sigma\varphi'(\sigma)}{\varphi(\sigma)}-
\dfrac{(\tau/\sigma)\widetilde\varphi'(\tau/\sigma)}{\widetilde\varphi(\tau/\sigma)}
\Big)
\varphi(\sigma)\widetilde\varphi\left(\frac{\tau}{\sigma}\right)
 \dfrac{d\sigma}{\sigma}
 \\
 &= \varphi(1)\widetilde\varphi(\tau) + \int\limits_1^{\sqrt{\tau}}\Big(
1 + \dfrac{\sigma\varphi'(\sigma)}{\varphi(\sigma)}
\Big)
\varphi(\sigma)\widetilde\varphi\left(\frac{\tau}{\sigma}\right)
 \dfrac{d\sigma}{\sigma}
 \\
&\qquad \qquad\quad- \int\limits_1^{\sqrt{\tau}}
\varphi(\sigma)\Big( 1 +
\dfrac{(\tau/\sigma)\widetilde\varphi'(\tau/\sigma)}{\widetilde\varphi(\tau/\sigma)}
\Big)\widetilde\varphi\left(\frac{\tau}{\sigma}\right)
 \dfrac{d\sigma}{\sigma}
 \\
 &= o(h_{\varphi, \widetilde\varphi}(\tau)) +  h_{\varphi, \widetilde\varphi}(\tau)(1+o(1)) -  h_{\varphi, \widetilde\varphi}(\tau)(1+o(1))
 \\
 & = o(h_{\varphi, \widetilde\varphi}(\tau)) = o(H[\widetilde \varphi \widetilde s, \varphi s](\tau)), \quad \tau\to\infty.
\end{align*}
}
We used Proposition~\ref{propSVF2}, part 2, and Proposition~\ref{propSVF1}, part 3, when estimating the integrals above.

Using similar arguments and utilizing Proposition~\ref{SVFasymp} we obtain the estimate
$$
\int\limits_1^{\sqrt{\tau}}
\Big(
\dfrac{\sigma\varphi'(\sigma)}{\varphi(\sigma)}-
\dfrac{(\tau/\sigma)\widetilde\varphi'(\tau/\sigma)}{\widetilde\varphi(\tau/\sigma)}
\Big)
\varphi(\sigma)\widetilde\varphi\left(\frac{\tau}{\sigma}\right)
\widetilde s\left(\ln\frac{\tau}{\sigma}\right) s(\ln\sigma) \dfrac{d\sigma}{\sigma} =
o(H[\widetilde \varphi \widetilde s, \varphi s](\tau))
$$
as $\tau\to\infty$, which concludes the proof of the lemma.
\end{proof}

\paragraph{{\bf The case of coinciding periods.}} 
Let's consider the case, when functions $s$ and $\widetilde s$ have a common period ($T = \widetilde{T}$). Denote
$$
(s \star \widetilde{s})(\eta) := \dfrac{1}{T}\int\limits_0^T s(\eta - \lda)\widetilde s(\lda)\,d\lda.
$$
Note, that there exists a continuous derivative
\begin{equation}\label{sotimes}
(s \star \widetilde{s})'(\eta) =
\dfrac{1}{T}\int\limits_0^T s(\eta - \lda) d(\widetilde s(\lda))
= -\dfrac{1}{p}(s \star \widetilde{s})(\eta) + e^{-\eta/p}\dfrac{1}{T}\int\limits_0^T \varrho(\eta - \lda)\,d\widetilde{\varrho}(\lda).
\end{equation}
The fact, that it is continuous, follows from the continuousness of $\varrho$ and $\widetilde\varrho$.

\begin{ruslemma} 
Let $\int\limits_1^{\infty}\widetilde\varphi(\tau)\dfrac{d\tau}{\tau} = \infty$,  $s$ and $\widetilde s$ have a common period $T$. 
Then
$$
\int\limits_1^{\sqrt{\tau}} \varphi\left(\frac{\tau}{\sigma}\right) s\left(\ln\frac{\tau}{\sigma}\right)
\widetilde\varphi(\sigma) \widetilde s(\ln\sigma)\dfrac{d\sigma}{\sigma} \sim h_{\widetilde\varphi, \varphi}(\tau)
(s \star \widetilde{s})(\ln\tau), \quad \tau\to\infty.
$$
If, also, $\int\limits_1^{\infty}\varphi(\tau)\dfrac{d\tau}{\tau} = \infty$, then
$$
\int\limits_1^{\tau} \varphi\left(\frac{\tau}{\sigma}\right) s\left(\ln\frac{\tau}{\sigma}\right)
\widetilde\varphi(\sigma) \widetilde s(\ln\sigma)\dfrac{d\sigma}{\sigma} \sim (\varphi\ast \widetilde\varphi)(\tau)
(s \star \widetilde{s})(\ln\tau), \quad \tau\to\infty.
$$
\end{ruslemma}
\begin{proof}
For $e^{2(k-1)T} < \tau \leq e^{2kT}$ we conclude
\begin{align*}
\int\limits_1^{\sqrt\tau} \varphi\left(\frac{\tau}{\sigma}\right)& s\left(\ln\frac{\tau}{\sigma}\right)
\widetilde \varphi(\sigma) \widetilde s(\ln\sigma)\dfrac{d\sigma}{\sigma} \sim
\int\limits_1^{e^{kT}} \varphi\left(\frac{\tau}{\sigma}\right) s\left(\ln\frac{\tau}{\sigma}\right)
\widetilde \varphi(\sigma) \widetilde s(\ln\sigma)\dfrac{d\sigma}{\sigma}
\\
&=\sum_{j=0}^{k-1}\int\limits_1^{e^T} \varphi(e^{-jT}\cdot\frac{\tau}{\sigma}) s\left(\ln\frac{\tau}{\sigma}\right)
\widetilde \varphi(e^{jT}\sigma) \widetilde s(\ln\sigma)\dfrac{d\sigma}{\sigma}
\\
& =
\int\limits_1^{e^T} s\left(\ln\frac{\tau}{\sigma}\right) \widetilde s(\ln\sigma)\sum_{j=0}^{k-1}\varphi(e^{-jT}\cdot\frac{\tau}{\sigma})\widetilde \varphi(e^{jT}\sigma)\dfrac{d\sigma}{\sigma}
\\
&=
\int\limits_1^{e^T} s\left(\ln\frac{\tau}{\sigma}\right) \widetilde s(\ln\sigma)T^{-1}\int\limits_1^{e^{kT}}
F_{e^{-(k-1)T}\cdot\frac{\tau}{\sigma}}[\varphi](e^{kT}/\xi)F_{\sigma}[\widetilde\varphi](\xi)\dfrac{d\xi}{\xi}\dfrac{d\sigma}{\sigma},
\end{align*}
where operator $F$ is introduced in \eqref{defF}. Considering the asymptotics~\eqref{asympF} and Proposition~\ref{propSVF2}, part 2, we obtain
$$
\int\limits_1^{\sqrt\tau} \varphi\left(\frac{\tau}{\sigma}\right) s\left(\ln\frac{\tau}{\sigma}\right)
\widetilde \varphi(\sigma) \widetilde s(\ln\sigma)\dfrac{d\sigma}{\sigma} \sim h_{\widetilde\varphi, \varphi}(e^{kT})(s \star \widetilde{s})(\ln\tau)\sim h_{\widetilde\varphi, \varphi}(\tau)(s \star \widetilde{s})(\ln\tau).
$$
The second part of the lemma could be proven similarly, if we consider the relation $\int_1^{\infty}\varphi(\tau)\dfrac{d\tau}{\tau} = \infty$.
\end{proof}

\begin{ruslemma}\label{raznos} Let $\int_1^{\infty}\widetilde\varphi(\tau)\dfrac{d\tau}{\tau} = \infty$,  $s$ and $\widetilde s$ have a common period $T$. 
Then
$$
H[\varphi s, \widetilde \varphi \widetilde s](\tau) \sim h_{\widetilde\varphi, \varphi}(\tau)
\Big(
\dfrac{1}{p}(s \star \widetilde{s})+(s \star \widetilde{s})'
\Big)(\ln\tau), \quad \tau\to\infty.
$$
If, also, $\int\limits_1^{\infty}\varphi(\tau)\dfrac{d\tau}{\tau} = \infty$, then
$$
(\varphi s \ast \widetilde \varphi \widetilde s)(\tau) \sim (\varphi\ast \widetilde\varphi)(\tau)
\Big(
\dfrac{1}{p}(s \star \widetilde{s})+(s \star \widetilde{s})'
\Big)(\ln\tau), \quad \tau\to\infty.
$$
\end{ruslemma}

\begin{proof}
Exactly the same proof as for the previous lemma. We just need to verify, that
$$
\int\limits_1^{e^T} s\left(\ln\frac{\tau}{\sigma}\right) \widetilde s(\ln\sigma)\dfrac{d(\widetilde{\varrho}(\ln \sigma))}{\widetilde{\varrho}(\ln \sigma)} =
\tau^{-1/p}\int\limits_0^T \varrho(\ln\tau - \lda)\,d\widetilde{\varrho}(\lda),
$$
which is clear, if we substitute $\lda = \ln\sigma$ in the left part.
\end{proof}

\paragraph{{\bf The case of incommensurable periods.}} Now, let the functions $s$ and $\widetilde s$ have no common period. 

\begin{ruslemma}\label{MellinConv3}
If periods $T$ and $\widetilde{T}$ are incommensurable, then 
$$
\int\limits_1^\tau s(\ln(\omega/\sigma)) \widetilde{s}(\ln\sigma)\dfrac{d(\widetilde{\varrho}(\ln \sigma))}{\widetilde{\varrho}(\ln \sigma)} = 
(\Cspec+o(1))\ln\tau, \qquad \tau\to+\infty
$$
uniformly over $\omega\in \mathbb{R}$, where 
\begin{equation}\label{Cdef}
\Cspec = \dfrac{1}{p}\cdot\dfrac{1}{T}
\int\limits_0^{T} s(t)\,dt \cdot \dfrac{1}{\widetilde{T}}
\int\limits_0^{\widetilde{T}} \widetilde{s}(t)\,dt.
\end{equation}
\end{ruslemma}

\begin{proof}
\emph{Step $1$.} Let's prove the estimate
\begin{equation}\label{Step1}
\int\limits_1^\tau s(\ln(\tau/\sigma)) \widetilde{s}(\ln\sigma)\dfrac{d(\widetilde{\varrho}(\ln \sigma))}{\widetilde{\varrho}(\ln \sigma)} = 
(\Cspec+o(1))\ln\tau, \qquad \tau\to+\infty.
\end{equation}
Substitute $t=\ln\tau$, $r=\ln\sigma$.
$$
\int\limits_1^\tau s(\ln(\tau/\sigma)) \widetilde{s}(\ln\sigma)\dfrac{d(\widetilde{\varrho}(\ln \sigma))}{\widetilde{\varrho}(\ln \sigma)} = \int\limits_0^t s(t-r)e^{-r/p}d\widetilde{\varrho}(r) =: Q(t).
$$
Define a $\widetilde T$-periodic function 
$$
q(t) := \int\limits_0^{T} s(r)\widetilde{s}(t+T-r)\,dr = \int\limits_{t}^{t+T} s(t-r)\widetilde s(r) \,dr.
$$
Note, that there exists a continuous derivative
$$
q'(t) = \int\limits_0^{T} s(r)d\widetilde{s}(t+T-r) = \int\limits_{t}^{t+T} s(t-r)d\widetilde s(r) = -\dfrac{1}{p} \cdot q(t) + \int\limits_t^{t+T} s(t-r)e^{-r/p}d\widetilde{\varrho}(r).
$$
Thus,
\begin{equation*}
Q(t+T) - Q(t) = \int\limits_t^{t+T} s(t-r)e^{-r/p}d\widetilde{\varrho}(r) = q'(t) + \dfrac{1}{p} \cdot q(t) =: q_1(t),
\end{equation*}
where $q_1(t)$ is a continuous $\widetilde{T}$-periodic function. Hence,
\begin{equation}\label{Qsum}
Q(t+n T) = Q(t) + \sum\limits_{k=0}^{n-1} q_1(t+k T).
\end{equation}
Using Oxtoby's ergodic theorem (see \cite{Oxtoby}) we obtain
\begin{equation}\label{Oxt}
\lim_{n\to+\infty} \dfrac{1}{n}\sum\limits_{k=0}^{n-1} q_1(t+k T) = \dfrac{1}{\widetilde{T}}
\int\limits_0^{\widetilde{T}} q_1(t)\,dt
\end{equation}
uniformly with respect to $t$. From \eqref{Qsum} and \eqref{Oxt}
we obtain the estimate
$$
Q(t) = (\Cspec+o(1))t, \qquad t\to+\infty,
$$
where
$$
\Cspec = \dfrac{1}{T\widetilde{T}}
\int\limits_0^{\widetilde{T}} q_1(t)\,dt = \dfrac{1}{p}\cdot\dfrac{1}{T}
\int\limits_0^{T} s(t)\,dt \cdot \dfrac{1}{\widetilde{T}}
\int\limits_0^{\widetilde{T}} \widetilde{s}(t)\,dt.
$$
Substitute $t=\ln\tau$, and the formula \eqref{Step1} is proven.
\vskip+0.5cm

\noindent \emph{Step $2$.} For every value of $\tau$ we could select $k(\tau)\in \mathbb{Z}$, such that $$
0 \leq \tau - \omega - Tk(\tau) < T.
$$
Then
\begin{multline*}
\int\limits_1^\tau s(\ln(\omega/\sigma)) \widetilde{s}(\ln\sigma)\dfrac{d(\widetilde{\varrho}(\ln \sigma))}{\widetilde{\varrho}(\ln \sigma)} =
\int\limits_{\omega+Tk(\tau)}^\tau s(\ln(\omega/\sigma)) \widetilde{s}(\ln\sigma)\dfrac{d(\widetilde{\varrho}(\ln \sigma))}{\widetilde{\varrho}(\ln \sigma)} 
\\
+ \int\limits_1^{\omega+Tk(\tau)} s(\ln((\omega+Tk(\tau))/\sigma)) \widetilde{s}(\ln\sigma)\dfrac{d(\widetilde{\varrho}(\ln \sigma))}{\widetilde{\varrho}(\ln \sigma)}.
\end{multline*}
The first term is uniformly bounded, and the second one satisfies the estimate
\begin{equation*}
(\Cspec+o(1))\ln(\omega+Tk(\tau)) = (\Cspec+o(1))\ln\tau, \qquad \tau\to+\infty.\qedhere
\end{equation*}
\end{proof}
\section{Spectral asymptotics of tensor products}
\begin{ruslemma}\label{ogrVar}
In the formula \eqref{Tasymp} function $s$ has the form
 $$
s(\tau) = e^{-\tau/p} \varrho(\tau),
$$
where $\varrho$ is a monotonous function, and thus $s$ is a bounded variation function.
\end{ruslemma}
\begin{proof}
The asymptotics could be transformed the following way:
$$
\dfrac{s(\ln(1/t))}{t^{1/p}} = \dfrac{\mathcal{N}(t)}{\varphi(1/t)}(1+\varepsilon(t)), \quad 
\varepsilon(t) \to 0 \text{ as } t\to +0.
$$
Replacing $t$ with $e^{-kT}t$ we obtain
$$
\dfrac{s(\ln(1/t))}{(e^{-kT}t)^{1/p}} = \dfrac{\mathcal{N}(e^{-kT}t)}{\varphi(e^{kT}/t)}(1+\varepsilon(e^{-kT}t)).
$$
Thus,
$$
\dfrac{s(\ln(1/t))}{t^{1/p}} = \lim_{k\to+\infty} e^{-\frac{kT}{p}}\dfrac{\mathcal{N}(e^{-kT}t)}{\varphi(e^{kT}/t)},
$$
and the convergence is uniform on $[1, e^T]$. Hence, for a fixed $\varepsilon > 0$ we obtain the relation
$$
s(\ln(1/t)) = t^{\frac{1}{p}+\varepsilon} \cdot
\lim_{k\to+\infty} \dfrac{e^{-\frac{kT}{p}-kT\varepsilon} \mathcal{N}(e^{-kT}t)}
{(e^{kT}/t)^{-\varepsilon}\varphi(e^{kT}/t)}.
$$
Note, that the numerator of the fraction decreases with $t$, and denominator increases with $t$ at large enough values of $k$, according to Proposition~\ref{propSVF1}, part 2. Denote
$$
\varrho_\varepsilon(\ln(1/t)) := \lim_{k\to+\infty} 
\dfrac{e^{-\frac{kT}{p}-kT\varepsilon} \mathcal{N}(e^{-kT}t)}
{(e^{kT}/t)^{-\varepsilon}\varphi(e^{kT}/t)}.
$$
As a uniform limit of monotonous functions, $\varrho_\varepsilon$ is monotonous. Function~$s$ has the form
$$
s(\tau) = e^{-(\frac{1}{p}+\varepsilon) \tau}\varrho_\varepsilon(\tau).
$$
Going to the limit as $\varepsilon\to 0$ and denoting $\varrho(\tau) := \lim\limits_{\varepsilon\to 0}\varrho_\varepsilon(\tau)$, we obtain
$$
s(\tau) = e^{-\tau/p} \varrho(\tau),
$$
where $\varrho$ is also a monotonous function.
\end{proof}

\begin{rusremark}\label{singmonot}
For some Green integral operators with singular arithmetically selfsimilar weight measures (see \cite{VladSheip,Vlad,Rast}) it is shown, that $\varrho(\tau)$ is a continuous purely singular function, i.e. its generalized derivative is a measure singular with respect to Lebesgue measure.

Below we assume, that all periodic functions arising in our asymptotics are continuous (thus satisfying all the requirements of \S 3). Also, according to Proposition~\ref{propSVF1}, part 3, we assume all SVFs to be $C^2$-smooth.
\end{rusremark}

\begin{rustheorem}\label{Th1}
Let operator $\mathcal{T}$ in a Hilbert space $\mathcal H$ have the spectral asymptotics \eqref{Tasymp}, and operator $\widetilde{\mathcal{T}}$ in a Hilbert space $\widetilde{\mathcal H}$ have the spectral asymptotics
$$
\widetilde{\mathcal{N}}(t) := \widetilde{\mathcal{N}}(t, \widetilde{\mathcal{T}}) = O(t^{-1/\widetilde{p}}), \quad t \to 0+,
\qquad \widetilde{p}>p.
$$
Then the operator $\mathcal{T}\otimes\widetilde{\mathcal{T}}$ in the Hilbert space $\mathcal{H}\otimes\widetilde{\mathcal{H}}$ has the asymptotics
\begin{equation}\label{Tasymp1}
\mathcal{N}_\otimes(t) := \mathcal{N}(t, \mathcal{T}\otimes\widetilde{\mathcal{T}}) \sim
\dfrac{\varphi(1/t)\cdot s^*(\ln(1/t))}{t^{1/p}}, \quad t\to+0, 
\end{equation}
where 
\begin{equation}\label{s_ast}
s^*(\tau) := \sum_{k} s(\tau + \ln(\widetilde{\lda}_k))\cdot\widetilde\lda_k^{1/p}
\end{equation}
is a periodic function with period $T$ {\rm(}the series converges, since $\widetilde{p}>p${\rm)}.
\end{rustheorem}
\begin{proof}
Since the eigenvalues of a tensor product of operators are equal to the products of their eigenvalues, we have
$$
\mathcal{N}_\otimes(t) = \#\{k, j: \lda_k \widetilde\lda_j > t\} =
\sum_k\#\{j: \lda_j > t/\widetilde\lda_k \} = \sum_k \mathcal{N}(t/\widetilde\lda_k).
$$
Thus,
$$
\dfrac{t^{1/p}}{\varphi(1/t)}\sum_k \mathcal{N}(t/\widetilde{\lda}_k) =
\sum_k \bigg( 
\dfrac{(t/\widetilde{\lda}_k)^{1/p} \mathcal{N}(t/\widetilde{\lda}_k)}
{\varphi(\widetilde{\lda}_k/t) s(\ln(\widetilde{\lda}_k/t))} 
\bigg)
\bigg(
\dfrac{\varphi(\widetilde{\lda}_k/t)}{\varphi(1/t)}
\bigg)
s(\ln(\widetilde{\lda}_k/t)) 
\widetilde{\lda}_k^{1/p}.
$$
The first multiplier is uniformly bounded and tends to $1$ as $t\to 0+$ according to
\eqref{Tasymp}. The second multiplier also tends to $1$. Also, since for every $\varepsilon$
function $\tau^\varepsilon\varphi(\tau)$ increases when $\tau > \tau_0(\varepsilon)$ according to Proposition~\ref{propSVF1}, part 2, we have an estimate
$$
\dfrac{\lda^\varepsilon\varphi(\lda\tau)}{\varphi(\tau)} =
\dfrac{(\lda\tau)^\varepsilon\varphi(\lda\tau)}{\tau^\varepsilon\varphi(\tau)}
\leq 1 \quad \text{ for } \lda\tau > \tau_0(\varepsilon), \lda < 1.
$$
Thus, for every $\varepsilon > 0$ we have an estimate (uniformly for $t < 1$)
$$
\dfrac{\varphi(\widetilde\lda_k/t)}{\varphi(1/t)}\leq C(\varepsilon)\widetilde\lda_k^{-\varepsilon},
$$
hence
$$
\dfrac{\varphi(\widetilde\lda_k/t)}{\varphi(1/t)}\widetilde\lda_k^{1/p} \leq
C(\varepsilon)\cdot k^{-\widetilde p (1/p - \varepsilon)},
$$
which, for sufficiently small $\varepsilon$ (such that $\widetilde p (1/p - \varepsilon) > 1$), gives us an estimate sufficient to use Lebesgue's monotone convergence theorem. Going to limit we obtain \eqref{Tasymp1}.
\end{proof}
\begin{rusremark}
For an arbitrary function $s$ and an operator $\widetilde{\mathcal{T}}$ function 
$s^\ast(\tau)$, generally speaking, could degenerate into constant. We could, for example, demand $s(\tau)+s(\tau+T/2)=1$, $T = 2p\ln2$, and choose a finite-rank operator $\widetilde{\mathcal{T}}$ with three eigenvalues: $2^{p}$, $2^{p}$ and $2^{2p}$. Then
\begin{align*}
s^\ast(\tau)& = 
s(\tau+p\ln 2)\cdot 2 + s(\tau+p\ln 2)\cdot 2 + 
 s(\tau + 2p\ln 2) \cdot 2^2 \\
 &= 4(s(\tau)+s(\tau+T/2)) = \const.
\end{align*}

However, if $s(\tau) = \exp(-\tau/p)\varrho(\tau)$, where $\varrho(\tau)$ is non-decreasing purely singular function (like in Remark~\ref{singmonot}), then no linear combination of shifts will be constant. Moreover, we note, that in this case function  $s^\ast(\tau)$ also has the form
$$
s^\ast(\tau) = \exp(-\tau/p)\varrho^\ast(\tau), \quad \varrho^\ast(\tau) = \sum_k\varrho(\tau+\ln\widetilde\lda_k),
$$
 and $\varrho^\ast(\tau)$ is a purely singular function, since $\varrho(\tau)$ is monotonous.
\end{rusremark}

Now, let's consider the case, when operators have coinciding power exponents in their spectral asymptotics.

\begin{rustheorem}\label{ThNEstim}
Let operator $\mathcal{T}$ have the spectral asymptotics \eqref{Tasymp}, and operator $\widetilde{\mathcal{T}}$ have the asymptotics
\begin{equation}\label{Tasymp2}
\mathcal{N}(t, \widetilde{\mathcal{T}}) \sim \widetilde{\mathcal{N}}_{as}(t) := 
\dfrac{\widetilde{\varphi}(1/t)\cdot \widetilde{s}(\ln(1/t))}{t^{1/p}}, \qquad t\to+0.
\end{equation}
Here $\widetilde\varphi$ is a SVF, $\widetilde s$ has period $\widetilde T$.
Then for every $\varepsilon > 0$ the estimates
\begin{equation*}\label{NEstim}
\mathcal{N}_\otimes(t) \lessgtr \dfrac{\alpha_\pm(\varepsilon)}{t^{1/p}}\cdot \bigg[
S(t,\varepsilon)+\widetilde S(t, \varepsilon) +\!\!\!\!\!
\int\limits_{\alpha_\mp(\varepsilon)/\varepsilon}^{\varepsilon\tau}\!\!\!\!\!
\varphi\left(\frac{\tau}{\sigma}\right)\widetilde\varphi(\sigma) s\left(\ln\frac{\tau}{\sigma}\right)
\widetilde s(\ln\sigma)\dfrac{d(\widetilde{\varrho}(\ln \sigma))}{\widetilde{\varrho}(\ln \sigma)}
\bigg]
\end{equation*}
hold uniformly for $t>0$. Here the integral should be interpreted as a Lebesgue-Sieltjes integral, $\tau = \alpha_\pm(\varepsilon)/t$. When $\varepsilon\tau < a_\mp(\varepsilon)/\varepsilon$ the integral is assumed to be zero. Coefficients $\alpha_\pm(\varepsilon)\to 1$ as $\varepsilon\to 0$, and functions $S(t, \varepsilon)$, $\widetilde S(t, \varepsilon)$ have the following asymptotics as $t\to+0$:
$$
S(t, \varepsilon) \sim \varphi(1/t)\cdot\sum_{\widetilde\lda_k\geqslant\varepsilon}s(\ln(1/t) + \ln(\widetilde\lda_k))\widetilde\lda_k^{1/p},
$$
\begin{equation}\label{asymptildeS}
\widetilde S(t, \varepsilon) \sim \widetilde\varphi(1/t)\cdot \Big(
\sum_{\lda_k \geqslant \varepsilon} \widetilde s(\ln(\tau)+\ln(\lda_k))\lda_k^{1/p}+
\varphi(1/\varepsilon)s(\ln(1/\varepsilon))\widetilde s(\ln(\tau\varepsilon))
\Big).
\end{equation}
\end{rustheorem}
\begin{proof}
The proof follows the scheme of Theorem 3.3 from \cite{NazNikKar}.
Let's prove the upper estimate, the lower estimate could be obtained similarly.
$$
t^{1/p}\mathcal{N}_\otimes(t) = t^{1/p}\sum_{\widetilde\lda_k < \varepsilon}
\mathcal{N}(t/\widetilde\lda_k) + S(t, \varepsilon),
$$
where
$$
S(t, \varepsilon) = t^{1/p}\sum_{\widetilde\lda_k\geqslant \varepsilon}\mathcal{N}(t/\widetilde\lda_k).
$$
The asymptotics for $S(t, \varepsilon)$ could be obtained from Theorem~\ref{Th1} for a finite-rank operator $\widetilde{\mathcal{T}}$.

Denote by $\widetilde \mu$ the inverse function to $\widetilde{\mathcal{N}}_{as}$. Then $\widetilde\lda_k/\widetilde\mu(k)\to 1$ as $k\to\infty$, thus
$$
\alpha_-(\varepsilon)\widetilde\mu(k) \leq \widetilde\lda_k \leq
\alpha_+(\varepsilon)\widetilde\mu(k) \quad\text{ for } \widetilde\lda_k < \varepsilon
$$
for certain values of $\alpha_\pm(\varepsilon)$, such that $\alpha_\pm(\varepsilon)\to 1$ as $\varepsilon\to 0$.

Function $\mathcal{N}$ is monotonous, which implies
$$
\sum_{\widetilde\lda_k < \varepsilon}
\mathcal{N}(t/\widetilde\lda_k) \leq
\sum_{\widetilde\mu(k) < \alpha_-^{-1}(\varepsilon)\varepsilon}
\mathcal{N}\Big(\frac{t}{\alpha_+(\varepsilon)\widetilde\mu(k)}\Big).
$$
Function $k\mapsto \mathcal{N}\big(\frac{t}{\alpha_+(\varepsilon)\widetilde\mu(k)}\big)$ is also monotonous, so we obtain
\begin{equation}\label{sumintest}
t^{1/p}\!\sum_{\widetilde\lda_k < \varepsilon}\!
\mathcal{N}(t/\widetilde\lda_k)\! \leq\! 
t^{1/p}\mathcal{N}\Big(\dfrac{\alpha_-(\varepsilon)t}{\alpha_+(\varepsilon)\varepsilon}\Big)
\!\!+t^{1/p}\!\!\!\int\limits_0^{\varepsilon\alpha_-^{-1}(\varepsilon)}\!\!\!
\mathcal{N}\Big(\frac{t}{\alpha_+(\varepsilon)\mu}\Big)
(-d\widetilde{\mathcal{N}}_{as}(\mu)).
\end{equation}
The first term could be estimated as $O(\varepsilon^{1/p}\varphi(1/t))$, thus, adding it to the term $S(t, \varepsilon)$, we obtain $\alpha_+(\varepsilon)S(t, \varepsilon)$. Further, considering $-d\widetilde{\mathcal{N}}_{as}(\mu)$ as a Lebesgue-Stieltjes measure, we obtain
\begin{equation}\label{Nas_measure}
-d\widetilde{\mathcal{N}}_{as}(\mu) = \dfrac{1}{\mu}\widetilde{\varphi}(1/\mu)\widetilde{\varrho}(\ln(1/\mu))
\Big(
\dfrac{-\mu d(\widetilde{\varrho}(\ln(1/\mu)))}{\widetilde{\varrho}(\ln(1/\mu))} +
\dfrac{\widetilde\varphi\,'(1/\mu)}{\mu\widetilde\varphi(1/\mu)} d\mu
\Big).
\end{equation}
The density of the second term tends to zero as $\mu\to 0$, while the first term
$$
\dfrac{-\mu d(\widetilde{\varrho}(\ln(1/\mu)))}{\widetilde{\varrho}(\ln(1/\mu))} = \dfrac{d\widetilde{\varrho}(\ln(1/\mu))}{\widetilde{\varrho}(\ln(1/\mu))} = d(\ln(\widetilde{\varrho}(\lda)))
$$ 
is a positive periodic measure (here $\lda = \ln(1/\mu)$), since
$$
\ln(\varrho(\tau+T)) = \ln(\varrho(\tau)) + \dfrac{T}{p}.
$$ 
Hence, for sufficiently small $\varepsilon$ the contribution of the second term of \eqref{Nas_measure} into the integral in \eqref{sumintest} is negligible, and this integral could be estimated as
$$
\alpha_+(\varepsilon)t^{1/p}\int\limits_0^{\varepsilon\alpha_-^{-1}(\varepsilon)}
\mathcal{N}\Big(\frac{t}{\alpha_+(\varepsilon)\mu}\Big)
 \widetilde\varphi(1/\mu)\big(-d(\widetilde{\varrho}(\ln(1/\mu)))\big).
$$
Splitting the integral into two parts and integrating by substitution, we obtain the estimate
$$
\alpha_+(\varepsilon)t^{1/p}\int\limits_\varepsilon^{+\infty}
\mathcal{N}(s)\widetilde\varphi(\tau s)d(\widetilde{\varrho}(\ln(\tau s))) +\alpha_+(\varepsilon)t^{1/p}\int\limits_{\alpha_-(\varepsilon)/\varepsilon}^{\varepsilon\tau}
\mathcal{N}(\sigma/\tau)\widetilde\varphi(\sigma)d(\widetilde{\varrho}(\ln\sigma)).
$$
Substituting in the second integral $\mathcal{N}$ with $\alpha_+(\varepsilon)\mathcal{N}_{as}$, we obtain exactly the third term of the estimate we are proving. The first integral gives us the term $\widetilde S(t, \varepsilon)$. Further,
$$
\dfrac{\widetilde\varphi(\alpha_+(\varepsilon)s/t)}{\widetilde\varphi(1/t)} \to 1 \quad
\text{ as } t\to 0
$$
uniformly for $s\in[\varepsilon, \lda_1(\mathcal{T})]$. Thus, 
$$
\widetilde S(t, \varepsilon) \sim \widetilde\varphi(1/t)\int\limits_\varepsilon^{+\infty}
\mathcal{N}(s)d(\widetilde s(\ln(\tau s)) s^{1/p}).
$$
It is clear, that $\mathcal{N}(s) = 0$ for $s>\lda_1(\mathcal{T})$. Integrating by parts, we obtain the asymptotics \eqref{asymptildeS}.
\end{proof}

In Theorems 3--5 we assume, that
\begin{equation}\label{phiinfinf}
\int\limits_1^{\infty}\varphi(\tau)\dfrac{d\tau}{\tau} = 
\int\limits_1^{\infty}\widetilde\varphi(\tau)\dfrac{d\tau}{\tau} = \infty.
\end{equation}

\begin{rustheorem}\label{Thinfinf}
Let operators $\mathcal{T}$ and $\widetilde{\mathcal{T}}$ satisfy the conditions of Theorem~{\rm\ref{ThNEstim}}.
Suppose, also, that condition \eqref{phiinfinf} holds, and the periods of $s$ and $\widetilde s$ coincide and equal $T$. Then
$$
\mathcal{N}_\otimes(t) \sim \dfrac{\phi(1/t)\cdot s_\otimes(\ln(1/t))}{t^{1/p}}, \quad t\to +0,
$$
where $\phi(s) := (\varphi \ast \widetilde\varphi)(s)$ is a SVF, 
\begin{equation}\label{sotimesdef}
s_\otimes(\eta) = \dfrac{(s \star \widetilde{s})(\eta)}{p} + (s \star \widetilde{s})'(\eta) = e^{-\eta/p}\dfrac{1}{T}\int\limits_0^T \varrho(\eta - \sigma)\,d\widetilde{\varrho}(\sigma)
\end{equation} is a continuous positive $T$-periodic function.
\end{rustheorem}
\begin{proof}
Fix $\varepsilon>0$ and consider the estimate obtained in Theorem~\ref{ThNEstim}.
According to Proposition~\ref{propSVF2}, part 1, we have
$$
S(t, \varepsilon) = o(\phi(1/t)), \quad
\widetilde S(t, \varepsilon) = o(\phi(1/t)), \qquad t\to +0.
$$
Further, we can extend the integration interval, since, considering $\tau = \alpha_{\pm}(\varepsilon)/t$ and using Proposition~\ref{propSVF2}, part 1, we have
\begin{multline*}
\int\limits_{\varepsilon\tau}^{\tau}
\varphi\left(\frac{\tau}{\sigma}\right)\widetilde\varphi(\sigma) s\left(\ln\frac{\tau}{\sigma}\right)
\widetilde s(\ln\sigma)\dfrac{d(\widetilde{\varrho}(\ln \sigma))}{\widetilde{\varrho}(\ln \sigma)} 
\\
\sim \widetilde \varphi(\tau)\int\limits_1^{1/\varepsilon}\varphi(\sigma)
 s(\ln\sigma)
\widetilde s(\ln(\tau/\sigma))\dfrac{d(\widetilde{\varrho}(\ln (\tau/\sigma)))}{\widetilde{\varrho}(\ln (\tau/\sigma))} = o(\phi(1/t)), \quad t\to+0,
\end{multline*}
\begin{multline*}
\int\limits_{1}^{\alpha_\mp(\varepsilon)/\varepsilon}
\varphi\left(\frac{\tau}{\sigma}\right)\widetilde\varphi(\sigma) s\left(\ln\frac{\tau}{\sigma}\right)
\widetilde s(\ln\sigma)\dfrac{d(\widetilde{\varrho}(\ln \sigma))}{\widetilde{\varrho}(\ln \sigma)} 
\\
\sim \varphi(\tau)\int\limits_1^{\alpha_\mp(\varepsilon)/\varepsilon}\widetilde\varphi(\sigma)
 s\left(\ln\frac{\tau}{\sigma}\right)
\widetilde s(\ln\sigma)\dfrac{d(\widetilde{\varrho}(\ln \sigma))}{\widetilde{\varrho}(\ln \sigma)} = o(\phi(1/t)), \quad t\to+0.
\end{multline*}
Thus
\begin{equation}\label{NisMellin}
\mathcal{N}_\otimes(t) \lessgtr \dfrac{\alpha_\pm(\varepsilon)}{t^{1/p}} (\varphi s \ast \widetilde \varphi \widetilde s)(\tau)(1+o(1)).
\end{equation}
Using Lemma~\ref{raznos}, we obtain 
$$
\mathcal{N}_\otimes(t) \lessgtr \alpha_\pm(\varepsilon) \dfrac{\phi(\tau)}{t^{1/p}}\Big( 
\dfrac{(s \star \widetilde{s})}{p} + (s \star \widetilde{s})'
\Big)(\ln(\tau))(1+o(1)), \quad t\to+0.
$$
Note, also, that $\phi(\tau) = \phi(1/t)(1+o(1))$ as $t\to+0$.
Hence
\begin{align}
\begin{split}\label{limsupliminf}
\limsup_{t\to+0}\mathcal{N}_\otimes(t) \Big( \dfrac{\phi(1/t)\cdot s_\otimes(\ln(1/t))}{t^{1/p}}\Big)^{-1}\!\!\! &\leq \alpha_+(\varepsilon) \cdot
\!\!\!\sup\limits_{t\in[1, e^T]}\dfrac{s_\otimes(\ln(\alpha_+(\varepsilon))\!+\ln(1/t))}{s_\otimes(\ln(1/t))}, 
\\
\liminf_{t\to+0}\mathcal{N}_\otimes(t) \Big( \dfrac{\phi(1/t)\cdot s_\otimes(\ln(1/t))}{t^{1/p}}\Big)^{-1} \!\!\!&\geq \alpha_-(\varepsilon)\cdot
\!\!\!\inf\limits_{t\in[1, e^T]}\dfrac{s_\otimes(\ln(\alpha_-(\varepsilon))\!+\ln(1/t))}{s_\otimes(\ln(1/t))}.
\end{split}
\end{align}
Function $s_\otimes$ is uniformly continuous on a segment, thus, supremum and infium in the right parts of \eqref{limsupliminf} tend to $1$ as $\varepsilon \to +0$. Going to the limit as $\varepsilon\to+0$ we conclude the proof.
\end{proof}
\begin{rusremark}
The question of non-constancy of $s_\otimes$ remains open. Even if we assume $s(\tau) = \exp(-\tau/p)\varrho(\tau)$, $\widetilde s(\tau) =
\exp(-\tau/p)\widetilde\varrho(\tau)$, and functions $\varrho$ and $\widetilde\varrho$ are purely singular, we have $s_\otimes(\tau) = \exp(-\tau/p)\varrho_\otimes(\tau)$, where
$$
\varrho_\otimes(\tau) = \dfrac{1}{T}\int\limits_0^T \varrho(\tau-\lda)d\widetilde\varrho(\lda).
$$
It is clear, that $\varrho_\otimes'=\varrho'\star\widetilde\varrho\,'$ is a convolution of singular measures. However, the convolution of singular measures often turns out to be absolutely continuous (see e.g.~\cite{DamanikGorodetskiSolomyak}).
\end{rusremark}

\begin{rustheorem}\label{noncommes}
Let operators $\mathcal{T}$ and $\widetilde{\mathcal{T}}$ satisfy the conditions of Theorem~{\rm \ref{ThNEstim}}.
Suppose, also, that condition \eqref{phiinfinf} holds, and and functions $s$ and $\widetilde s$ have no common period, i.e. their periods $T$ and $\widetilde T$ are incommensurable. Then
$$
\mathcal{N}_\otimes(t) \sim \dfrac{\psi(1/t)\phi(1/t)}{t^{1/p}}, \quad t\to +0,
$$
where $\phi(s) = (\varphi \ast \widetilde\varphi)(s)$, $\psi(t)$ is a certain bounded and separated from zero SVF.
\end{rustheorem}
\begin{proof}
We repeat the proof of Theorem \ref{Thinfinf} until we obtain the relation~\eqref{NisMellin}. Further, we obtain an estimate, that we could use instead of Lemma~\ref{raznos}.

We introduce a function $r(\tau)$, defined according to the relation
$$
(\varphi s \ast \widetilde \varphi \widetilde s)(\tau) = 
\phi(\tau) r(\ln\tau).
$$
Function $r$ is bounded and separated from zero according to Lemma~\ref{MellinConv1}. We need to prove, that it is uniformly continuous. We have
\begin{multline*}
r(\ln\tau + \delta) - r(\ln\tau) = r(\ln\tau+\delta) \Big(\dfrac{\phi(\tau e^\delta)}{\phi(\tau)}  - 1\Big)
\\ 
+ \dfrac{1}{\phi(\tau)}\cdot \int\limits_1^\tau \left(\dfrac{\varphi(\frac{\tau e^\delta}{\sigma}) s(\ln \frac{\tau e^\delta}{\sigma})}{\varphi\left(\frac{\tau}{\sigma}\right) s\left(\ln\frac{\tau}{\sigma}\right)} - 1\right) \varphi\left(\frac{\tau}{\sigma}\right) s\left(\ln\frac{\tau}{\sigma}\right)
\widetilde \varphi(\sigma) \widetilde s(\ln\sigma)\dfrac{d(\widetilde{\varrho}(\ln \sigma))}{\widetilde{\varrho}(\ln \sigma)}
\\
 + \dfrac{1}{\phi(\tau)}\cdot \int\limits_\tau^{\tau e^\delta} \varphi\left(\frac{\tau}{\sigma}\right) s\left(\ln\frac{\tau}{\sigma}\right)
\widetilde \varphi(\sigma) \widetilde s(\ln\sigma)\dfrac{d(\widetilde{\varrho}(\ln \sigma))}{\widetilde{\varrho}(\ln \sigma)}.
\end{multline*}
Let's show, that each term here tends to zero as $\delta\to 0$ uniformly with respect to $\tau$. Without loss of generality we assume, that $0<\delta\leq\delta_0$ for a certain value of ~$\delta_0$. 
For the first term we use the mean value theorem:
$$
\dfrac{\phi(\tau e^\delta) - \phi(\tau)}{\phi(\tau)} =
(\tau e^\delta-\tau) \dfrac{\phi'(\zeta)}{\phi(\tau)} = (e^\delta-1) \cdot \dfrac{\tau}{\zeta}\cdot\dfrac{\phi(\zeta)}{\phi(\tau)}\cdot\dfrac{\zeta\phi'(\zeta)}{\phi(\zeta)},
$$
where $\zeta \in [\tau, \tau e^\delta]$. Multiplier $\dfrac{\tau}{\zeta}$ is bounded. 
For the last two multipliers there exist the limits
$$
\dfrac{\phi(\zeta)}{\phi(\tau)} \to 1, \quad \dfrac{\zeta\phi'(\zeta)}{\phi(\zeta)} \to 0, \quad \tau\to\infty,
$$
which means that they are also bounded. Thus,
$$
\left| \dfrac{\phi(\tau e^\delta)}{\phi(\tau)}  - 1 \right| \leq
C (e^\delta - 1) \to 0, \quad \delta\to 0
$$
uniformly for $\tau \in \mathbb{R}_+$.

Similarly, it is possible to show, that in the second term the expression 
$$
\dfrac{\varphi(\frac{\tau e^\delta}{\sigma}) s(\ln \frac{\tau e^\delta}{\sigma})}{ \varphi\left(\frac{\tau}{\sigma}\right) s\left(\ln\frac{\tau}{\sigma}\right)} - 1,
$$
tends to zero uniformly, since $\varphi$ is a SVF, $s$ is continuous, periodic, bounded and separated from zero. The rest of the multipliers in the second term are bounded  according to Lemma~\ref{MellinConv1}.

For the third term similarly to Lemma~\ref{MellinConv1} we obtain an estimate
$$
\int\limits_\tau^{\tau e^\delta} \varphi\left(\frac{\tau}{\sigma}\right) s\left(\ln\frac{\tau}{\sigma}\right)
\widetilde \varphi(\sigma) \widetilde s(\ln\sigma)\dfrac{d(\widetilde{\varrho}(\ln \sigma))}{\widetilde{\varrho}(\ln \sigma)} = O(\phi(\tau e^\delta) - \phi(\tau)),
\quad \tau\to\infty,
$$
thus it tends to zero same as the first one. Thereby, the uniform continuousness is proven.

Now we need to prove, that $r(\ln\tau)$ is a SVF. By definition we have
\begin{equation}
\begin{split}\label{restim0}
&r(\ln\tau+T)\phi(\tau e^T) - r(\ln\tau)\phi(\tau)
\\
&\quad = \int\limits_\tau^{\tau e^T}\varphi\left(e^T\cdot \frac{\tau}{\sigma}\right) s\left(\ln\frac{\tau}{\sigma}\right)
\widetilde \varphi(\sigma) \widetilde s(\ln\sigma)\dfrac{d(\widetilde{\varrho}(\ln \sigma))}{\widetilde{\varrho}(\ln \sigma)}
\\
&\qquad+\int\limits_1^{\tau}\left(\varphi\left(e^T\cdot \frac{\tau}{\sigma}\right)-\varphi\left(\frac{\tau}{\sigma}\right)\right) s\left(\ln\frac{\tau}{\sigma}\right)
\widetilde \varphi(\sigma) \widetilde s(\ln\sigma)\dfrac{d(\widetilde{\varrho}(\ln \sigma))}{\widetilde{\varrho}(\ln \sigma)}.
\end{split}
\end{equation}
We rewrite the left-hand side of \eqref{restim0}, considering the relation $\phi(\tau e^T) = \phi(\tau)(1+o(1))$, as
\begin{equation}\label{restim1}
r(\ln\tau+T)\phi(\tau e^T) - r(\ln\tau)\phi(\tau) = \big(r(\ln\tau + T) - r(\ln\tau)\big)\phi(\tau) + o(\phi(\tau)),\quad\tau\to\infty.
\end{equation}
In the right-hand side of \eqref{restim0}, as $\tau\to\infty$, the first integral satisfies the estimate
\begin{align}
\begin{split}\label{restim2}
\int\limits_\tau^{\tau e^T}\varphi&\left(e^T\cdot \frac{\tau}{\sigma}\right) s\left(\ln\frac{\tau}{\sigma}\right)
\widetilde \varphi(\sigma) \widetilde s(\ln\sigma)\dfrac{d(\widetilde{\varrho}(\ln \sigma))}{\widetilde{\varrho}(\ln \sigma)}
 \\
 &\sim \widetilde\varphi(\tau)\int\limits_1^{e^T}\varphi(\frac{e^T}{\sigma}) s\left(\ln \frac{1}{\sigma}\right)
\widetilde s(\ln(\tau\sigma))\dfrac{d(\widetilde{\varrho}(\ln (\tau\sigma)))}{\widetilde{\varrho}(\ln (\tau\sigma))} = o(\phi(\tau)).
\end{split}
\end{align}
To estimate the second integral we use the Proposition~\ref{SVFasymp}. Since
$\varphi(\tau e^T) = \varphi(\tau)(1+o(1))$ as $\tau\to\infty$, we have
\begin{align*}
& \int\limits_1^{\tau}\varphi\left(e^T\cdot \frac{\tau}{\sigma}\right) s\left(\ln\frac{\tau}{\sigma}\right)
\widetilde \varphi(\sigma)\widetilde s(\ln\sigma)\dfrac{d(\widetilde{\varrho}(\ln \sigma))}{\widetilde{\varrho}(\ln \sigma)}
\\
&\quad = \int\limits_1^{\tau}\varphi\left(\frac{\tau}{\sigma}\right) s\left(\ln\frac{\tau}{\sigma}\right)
\widetilde \varphi(\sigma) \widetilde s(\ln\sigma)\dfrac{d(\widetilde{\varrho}(\ln \sigma))}{\widetilde{\varrho}(\ln \sigma)}(1+o(1)).
\end{align*}
According to Lemma~\ref{MellinConv1} the integral in the right-hand side could be estimated as $O(\phi(\tau))$, thus
\begin{equation}\label{restim3}
\int\limits_1^{\tau}\left(\varphi\left(e^T\cdot \frac{\tau}{\sigma}\right)-\varphi\left(\frac{\tau}{\sigma}\right)\right) s\left(\ln\frac{\tau}{\sigma}\right)
\widetilde \varphi(\sigma) \widetilde s(\ln\sigma)\dfrac{d(\widetilde{\varrho}(\ln \sigma))}{\widetilde{\varrho}(\ln \sigma)} = o(\phi(\tau)),
\quad \tau\to\infty.
\end{equation}
From \eqref{restim1}, \eqref{restim2}, \eqref{restim3} it follows, that
$$
r(\ln\tau + T) - r(\ln\tau) = o(1), \quad \tau\to\infty.
$$
By the same argument, we obtain
$$
r(\ln\tau + \widetilde T) - r(\ln\tau) = o(1), \quad \tau\to\infty.
$$
Hence, for arbitrary $z_1, z_2 \in \mathbb{Z}$ we have
$$
r(\ln\tau +z_1 T + z_2\widetilde T) - r(\ln\tau) = o(1), \quad \tau\to\infty.
$$
Since the periods are incommensurable, the set $\{z_1T+z_2\widetilde T | z_1, z_2 \in \mathbb{Z}\}$ is dense in $\mathbb{R}$, thus, from the uniform continuousness of $r$ it follows, that for every $c\in\mathbb{R}$ we have
$$
r(\ln\tau + c) - r(\ln\tau) = o(1), \quad \tau\to\infty.
$$
Hence, considering, that $r$ is bounded and separated from zero, we conclude, that the function $\psi(\tau) := r(\ln(\tau))$ is a SVF.
\end{proof}

\noindent Under certain additional conditions it is possible to demonstrate, that $\psi = \const$.

\begin{rustheorem}\label{noncommes1}
Let operators $\mathcal{T}$ and $\widetilde{\mathcal{T}}$ satisfy the conditions of Theorem~{\rm \ref{noncommes}}. Suppose, also, that functions $\varphi$ and $\widetilde\varphi$ satisfy the following estimates:
\begin{equation}\label{addconstrict}
\left|\dfrac{\sigma\ln(\sigma) \varphi'(\sigma)}{\varphi(\sigma)}\right|\leq C, \quad
\left|\dfrac{\sigma\ln(\sigma) \widetilde\varphi'(\sigma)}{\widetilde\varphi(\sigma)}\right|\leq C, \qquad \sigma \geqslant 1.
\end{equation}
Then
$$
\mathcal{N}_\otimes(t) \sim \dfrac{\Cspec\phi(1/t)}{t^{1/p}}, \quad t\to +0,
$$
where $\phi(s) = (\varphi \ast \widetilde\varphi)(s)$, and the constant $\Cspec$ is defined in \eqref{Cdef}.
\end{rustheorem}
\begin{proof}
We are aiming to prove the estimate
\begin{equation}\label{Th5Asymp1}
(\varphi s \ast \widetilde \varphi \widetilde s)(\tau) \sim 
\Cspec\phi(\tau), \quad \tau\to\infty.
\end{equation}
First, let's estimate $H[\varphi s, \widetilde \varphi \widetilde s](\tau)$. In order to do that, we integrate by parts and use Lemma~\ref{MellinConv3}.
{\allowdisplaybreaks
\begin{align*}
H[\varphi s, \widetilde \varphi \widetilde s](\tau)& =
\int\limits_1^{\sqrt{\tau}} \varphi\left(\frac{\tau}{\sigma}\right) \widetilde \varphi(\sigma) 
\,d\bigg(\int\limits_1^\sigma s\big(\ln \frac{\tau}{\xi}\big)
 \widetilde s\big(\ln\xi\big)\dfrac{d(\widetilde{\varrho}(\ln \xi))}{\widetilde{\varrho}(\ln \xi)}\bigg)
 \\
&  = \varphi(\sqrt{\tau}) \widetilde \varphi(\sqrt{\tau}) 
\int\limits_1^{\sqrt{\tau}} s\big(\ln \frac{\tau}{\xi}\big)
 \widetilde s\big(\ln\xi\big)\dfrac{d(\widetilde{\varrho}(\ln \xi))}{\widetilde{\varrho}(\ln \xi)}
 \\
 &\qquad - \int\limits_1^{\sqrt{\tau}} \Big(\varphi\left(\frac{\tau}{\sigma}\right) \widetilde \varphi(\sigma)\Big)'_\sigma 
\int\limits_1^\sigma s\big(\ln \frac{\tau}{\xi}\big)
 \widetilde s\big(\ln\xi\big)\dfrac{d(\widetilde{\varrho}(\ln \xi))}{\widetilde{\varrho}(\ln \xi)} \,d\sigma
 \\
 & = \varphi(\sqrt{\tau}) \widetilde \varphi(\sqrt{\tau}) 
(\Cspec+o(1))\ln(\sqrt{\tau})
\\
&\qquad - \int\limits_1^{\sqrt{\tau}} \Big(\varphi\left(\frac{\tau}{\sigma}\right) \widetilde \varphi(\sigma)\Big)'_\sigma 
(\Cspec+o(1))\ln\sigma \,d\sigma.
\end{align*}
}
We transform the main term of the asymptotics by reversing the integration by parts:
$$
\Cspec \Big(\varphi(\sqrt{\tau}) \widetilde \varphi(\sqrt{\tau}) 
\ln(\sqrt{\tau}) - \int\limits_1^{\sqrt{\tau}} \Big(\varphi\left(\frac{\tau}{\sigma}\right) \widetilde \varphi(\sigma)\Big)'_\sigma 
\ln\sigma \,d\sigma\Big) = \Cspec h_{\widetilde\varphi, \varphi}(\tau).
$$
Now, let's estimate the contribution of each $o(1)$.
\begin{multline*}
\Cspec h_{\widetilde\varphi, \varphi}(\tau) + \int\limits_1^{\sqrt{\tau}} \Big(\varphi\left(\frac{\tau}{\sigma}\right) \widetilde \varphi(\sigma)\Big)'_\sigma \ln\sigma \cdot o(1) \,d\sigma 
\\
 = \Cspec\!\int\limits_1^{\sqrt{\tau}}\!\! \Big[ 1 + \Big(\dfrac{\sigma\ln(\sigma) \widetilde\varphi'(\sigma)}{\widetilde
\varphi(\sigma)} + \dfrac{\ln(1/\sigma)}{\ln(\tau/\sigma)}\cdot\dfrac{(\tau/\sigma)\ln(\tau/\sigma) \varphi'(\tau/\sigma)}{\varphi(\tau/\sigma)}\Big) o(1)\Big] \cdot
\varphi\!\left(\frac{\tau}{\sigma}\right)\! \widetilde \varphi(\sigma\!) 
\dfrac{d\sigma}{\sigma} 
\\
= (\Cspec+o(1))h_{\widetilde\varphi, \varphi}(\tau).
\end{multline*}
according to Proposition~\ref{propSVF2}, part 2, since the expression in the round parentheses is bounded according to the additional conditions \eqref{addconstrict}.
By the same argument we have
\begin{align*}
&\varphi(\sqrt{\tau}) \widetilde \varphi(\sqrt{\tau}) 
\ln(\sqrt{\tau}) \cdot o(1)
\\
&\qquad = o(1) \cdot \Big(h_{\widetilde\varphi, \varphi}(\tau) + 
\int\limits_1^{\sqrt{\tau}} \Big(\varphi\left(\frac{\tau}{\sigma}\right) \widetilde \varphi(\sigma)\Big)'_\sigma \ln\sigma \,d\sigma \Big) = o(h_{\widetilde\varphi, \varphi}(\tau)).
\end{align*}

\noindent Thus, we obtained the estimate
\begin{equation}\label{Th5Asymp2}
H[\varphi s, \widetilde \varphi \widetilde s](\tau) =  (\Cspec+o(1)) h_{\widetilde\varphi, \varphi}(\tau).
\end{equation}
Similarly, considering Lemma~\ref{MellinConv2}, we obtain
\begin{equation}\label{Th5Asymp3}
H_1[\varphi s, \widetilde \varphi \widetilde s](\tau) = H[\widetilde \varphi \widetilde s, \varphi s](\tau)(1+o(1)) = (\Cspec+o(1))h_{\varphi, \widetilde\varphi}(\tau).
\end{equation}
From asymptotics \eqref{Th5Asymp2} and \eqref{Th5Asymp3} we obtain the required asymptotics \eqref{Th5Asymp1}.
\end{proof}

\begin{rusremark}
From the additional conditions \eqref{addconstrict} it follows, that for a certain $C > 0$ the estimates
$$
\varphi(e) (\ln\sigma)^{-C} \leq \varphi(\sigma) \leq \varphi(e) (\ln\sigma)^{C}, \qquad \widetilde \varphi(e) (\ln\sigma)^{-C} \leq \widetilde\varphi(\sigma) \leq \widetilde \varphi(e) (\ln\sigma)^{C}
$$
hold for $\sigma \geqslant e$. The additional conditions clearly hold for the SVFs of the form $(1+\ln(\tau))^{\varkappa}$. In the general case, the question of the constancy of the function $\psi$ in the Theorem~\ref{noncommes} remains open.
\end{rusremark}

Now let's consider the cases, when one or both of the integrals of the SVFs are finite.

\begin{rustheorem}\label{Thinfbound}
Let operators $\mathcal{T}$ and $\widetilde{\mathcal{T}}$ satisfy the conditions of Theorem~{\rm \ref{ThNEstim}}, suppose
$$
\int\limits_1^{\infty}\varphi(\tau)\dfrac{d\tau}{\tau} < \infty, \quad 
\int\limits_1^{\infty}\widetilde\varphi(\tau)\dfrac{d\tau}{\tau} = \infty,
$$
and the periods of the functions $s$ and $\widetilde s$ coincide and equal $T$.
Suppose, also, that for $(\varphi, \widetilde\varphi)$ part 4 of the Proposition~{\rm\ref{propSVF2}} holds. Then
$$
\mathcal{N}_\otimes(t)\sim \dfrac{h_{\widetilde\varphi, \varphi}(1/t)
\cdot s_\otimes(\ln(1/t))
 +\widetilde\varphi(1/t) \cdot \widetilde s^*(\ln(1/t))
}{t^{1/p}},
$$
where $s_\otimes$ is defined in \eqref{sotimesdef}, and
\begin{equation}\label{s_ast_tilde}
\widetilde s^*(\tau) = \sum_n \widetilde s(\tau+\ln(\lda_n))\lda_n^{1/p}
\end{equation}
{\rm(}cf. \eqref{s_ast}{\rm)}.
\end{rustheorem}
\begin{rusremark}
The sum \eqref{s_ast_tilde} converges according to the Proposition~\ref{proposEIGEN}.
\end{rusremark}
\begin{proof}

Fix $\varepsilon>0$. According to Proposition~\ref{propSVF2}, part 1, we have
$$
S(t, \varepsilon) = o(h_{\widetilde\varphi, \varphi}(1/t)), \quad t\to+0.
$$
According to Proposition~\ref{propSVF1}, part 4, we have
$$
\widetilde S(t, \varepsilon) \sim \widetilde\varphi(1/t)\cdot \Big(
\sum_{n} \widetilde s(\ln\tau+\ln\lda_k)\lda_k^{1/p}+
\nu(\varepsilon)
\Big),
$$
where $\nu(\varepsilon)\to 0$ as $\varepsilon\to +0$. What's left is to estimate the integral term.
{\allowdisplaybreaks
\begin{align*}
&\int\limits_{\alpha_\mp(\varepsilon)/\varepsilon}^{\varepsilon\tau}
\varphi\left(\frac{\tau}{\sigma}\right)\widetilde\varphi(\sigma) s\left(\ln\frac{\tau}{\sigma}\right)
\widetilde s(\ln\sigma)\dfrac{d(\widetilde{\varrho}(\ln \sigma))}{\widetilde{\varrho}(\ln \sigma)}
\\
&\quad= \int\limits_{1}^{\sqrt{\tau}}
\varphi\left(\frac{\tau}{\sigma}\right)\widetilde\varphi(\sigma) s\left(\ln\frac{\tau}{\sigma}\right)
\widetilde s(\ln\sigma)\dfrac{d(\widetilde{\varrho}(\ln \sigma))}{\widetilde{\varrho}(\ln \sigma)} 
\\
&\qquad-\int\limits_1^{\alpha_\mp(\varepsilon)/\varepsilon}
\varphi\left(\frac{\tau}{\sigma}\right)\widetilde\varphi(\sigma) s\left(\ln\frac{\tau}{\sigma}\right)
\widetilde s(\ln\sigma)\dfrac{d(\widetilde{\varrho}(\ln \sigma))}{\widetilde{\varrho}(\ln \sigma)} 
\\
&\qquad+\int\limits_{1/\varepsilon}^{\sqrt{\tau}}
\varphi(\sigma)\widetilde\varphi\left(\frac{\tau}{\sigma}\right) s(\ln\sigma)
\widetilde s(\ln(\tau/\sigma))\dfrac{d(\widetilde{\varrho}(\ln (\tau/\sigma)))}{\widetilde{\varrho}(\ln (\tau/\sigma))}.
\end{align*}}
We estimate the first term by Lemma~\ref{raznos}. We estimate the second term as $O(\varphi(\tau))=
o(h_{\widetilde\varphi, \varphi}(\tau))$ as $\tau\to\infty$. What's left is to estimate the third term, which we estimate similarly to Proposition~\ref{propSVF2}, part 4:
\begin{align*}
\int\limits_{1/\varepsilon}^{\sqrt{\tau}}
&\varphi(\sigma)\widetilde\varphi\left(\frac{\tau}{\sigma}\right) s(\ln\sigma)
\widetilde s(\ln(\tau/\sigma))\dfrac{d(\widetilde{\varrho}(\ln (\tau/\sigma)))}{\widetilde{\varrho}(\ln (\tau/\sigma))}
\leq C \widetilde\varphi(\tau)\int\limits_{1/\varepsilon}^{\infty}
\varphi(\sigma)\dfrac{d\sigma}{\sigma}, \quad
\tau\to\infty.
\end{align*}
Here
$$
\int\limits_{1/\varepsilon}^{\infty}
\varphi(\sigma)\dfrac{d\sigma}{\sigma} \to 0, \quad \varepsilon\to 0,
$$
thus, this term's contribution to the asymptotics is negligible.
\end{proof}

\begin{rusremark}
Similarly to the Theorem~\ref{noncommes}, if the periods $T$ and $\widetilde T$ are incommensurable, then instead of $s_\otimes(\ln(\tau))$ in the asymptotics we obtain a bounded and separated from zero SVF, degenerating into constant in the same particular cases as in the Theorem~\ref{noncommes1}.
\end{rusremark}

\begin{rustheorem}\label{Thboundbound}
Let operators $\mathcal{T}$ and $\widetilde{\mathcal{T}}$ satisfy the conditions of Theorem~{\rm\ref{ThNEstim}}, suppose
$$
\int\limits_1^{\infty}\varphi(\tau)\dfrac{d\tau}{\tau} < \infty, \quad 
\int\limits_1^{\infty}\widetilde\varphi(\tau)\dfrac{d\tau}{\tau} < \infty,
$$
and for $(\varphi, \widetilde\varphi)$ and $(\widetilde\varphi, \varphi)$ part 4 of the Poposition~{\rm\ref{propSVF2}} holds. Then
$$
\mathcal{N}_\otimes(t)\sim \dfrac{ \varphi(1/t) \cdot s^*(\ln(1/t))+\widetilde\varphi(1/t)
 \cdot \widetilde s^*(\ln(1/t))
}{t^{1/p}},
$$
where $s^*$ is defined in \eqref{s_ast}, $\widetilde s^*$ is defined in \eqref{s_ast_tilde}.
\end{rustheorem}
\noindent The proof of this theorem is similar to the previous one.

\begin{rusremark}
In contrast to the previous theorems, the asymptotics in the last two cases contain two terms. One of them might be majorized by the other, in that case the asymptotics is almost regular again. However, in the general case, it is impossible to predict their behavior, and it is possible, that neither prevails. In that case the asymptotics might not be almost regular.
\end{rusremark}

\begin{rusexample}\label{LogExample}
Let
$$
\mathcal{N}(t, {\mathcal{T}}) \sim  
\dfrac{\ln^{\varkappa_1}(1/t)\cdot s(\ln(1/t))}{t^{1/p}},\qquad
\mathcal{N}(t, \widetilde{\mathcal{T}}) \sim  
\dfrac{\ln^{\varkappa_2}(1/t)\cdot \widetilde{s}(\ln(1/t))}{t^{1/p}}
$$
as  $t\to+0$. Without loss of generality we assume, that $\varphi(\tau) = (1+\ln(\tau))^{\varkappa_1}$, $\widetilde\varphi(\tau) = (1+\ln(\tau))^{\varkappa_2}$. In this case, the asymptotics of the Mellin convolution was calculated in the Example 1 in \cite{NazNikKar}. Let's consider all possible cases.

\paragraph*{Case 1. $\varkappa_1 \geqslant -1, \varkappa_2 \geqslant -1$.}
In this case Theorem \ref{Thinfinf} is applicable when periodic functions have a common period, and Theorem \ref{noncommes1} is applicable otherwise.

If the functions $s$ and $\widetilde s$ have a common period $T$, then
$$
\mathcal{N}_\otimes(t) \sim \dfrac{\phi(1/t)\cdot s_\otimes(\ln(1/t))}{t^{1/p}}, \quad t\to +0,
$$
where the function $s_\otimes$ is defined in \eqref{sotimesdef}, 
$$
    \phi(\tau) =
    \begin{cases}
    \mathbf{B}(\varkappa_1+1, \varkappa_2+1) (1+\ln(\tau))^{\varkappa_1+\varkappa_2+1}, &  \varkappa_1 > -1, \varkappa_2 > -1, \\
    \ln(\ln(\tau))\cdot (1+\ln(\tau))^{\varkappa_2}, &  \varkappa_1 = -1, \varkappa_2 > -1, \\
    2\ln(\ln(\tau))\cdot (1+\ln(\tau))^{-1}, &  \varkappa_1 = \varkappa_2 = -1, \\
    \end{cases}
$$
where $\mathbf{B}$ is the Euler beta function. Note, that the resulting asymptotics is again almost regular.

If the periods $T$ and $\widetilde T$ are incommensurable, then
$$
\mathcal{N}_\otimes(t) \sim \dfrac{\Cspec\phi(1/t)}{t^{1/p}}, \quad t\to +0,
$$
where the constant $\Cspec$ is defined in \eqref{Cdef}, and the resulting asymptotics is regular.

\paragraph*{Case 2. $\varkappa_1 < -1 \leq \varkappa_2$.} In this case, Theorem \ref{Thinfbound} is applicable, and by direct calculation it is easy to see, that
$$
h_{\widetilde\varphi, \varphi}(\tau) = o(\widetilde \varphi(\tau)), \quad \tau\to\infty,
$$
which means, that
\begin{equation}\label{Notimes1}
\mathcal{N}_\otimes(t)\sim \dfrac{\ln^{\varkappa_2}(1/t)
\cdot \widetilde s^*(\ln(1/t))}{t^{1/p}},
\end{equation}
where $\widetilde s^*$ is defined in \eqref{s_ast_tilde}, and the resulting asymptotics is again almost regular.

\paragraph*{Case 3. $\varkappa_1 < \varkappa_2 < -1$.} In this case Theorem \ref{Thboundbound} is applicable, and 
$$
\varphi(\tau) = o(\widetilde \varphi(\tau)), \quad \tau\to\infty,
$$
thus, again, we have the asymptotics \eqref{Notimes1}.

\paragraph*{Case 4. $\varkappa_1 = \varkappa_2 < -1$.} In this case Theorem \ref{Thboundbound} is applicable, and both terms of the asymptotics have the same order of growth, thus
$$
\mathcal{N}_\otimes(t)\sim \dfrac{\ln^{\varkappa_1}(1/t)\big(
 s^*(\ln(1/t)) +
\widetilde s^*(\ln(1/t))\big)}{t^{1/p}},
$$
where $s^*$ is defined in \eqref{s_ast}, $\widetilde s^*$ is defined in \eqref{s_ast_tilde}. In the case, when the functions $s$ and $\widetilde s$ have a common period, this asymptotics turns out to be almost regular, but in the case, when the periods are incommensurable, we have an almost regular asymptotics with a quasi-periodic component.

\end{rusexample}

\section{Small deviations asymptotics}

Let us recall some facts from the theory of small deviations in $L_2$ of Gaussian random functions.

Let there be a Gaussian random function $X(x)$, $x \in \mathcal{O} \subseteq \mathbb{R}^m$, with zero mean and a covariation function $G_X(x, u)$, $x, u \in\mathcal{O}$.
Let $\mu$ be a finite measure on $\mathcal{O}$. Denote
$$
\|X\|_\mu = \Big( \int\limits_{\mathcal{O}}  X^2(x) d\mu(x)\Big)^{1/2}
$$
We call the logarithmic asymptotics of small deviations in $L_2$ the asymptotics of \mbox{$\ln{\bf P}\{\|X\|_\mu \leq \varepsilon\}$} as $\varepsilon\to 0$. 

According to the well-known Karhunen--Lo\`eve expansion we have in distribution
$$
\|X(x)\|^2_\mu \stackrel{d}{=} \sum_{n=1}^{\infty} \lambda_n\xi^2_n,
$$
where $\xi_n$, $n \in \mathbb{N}$, are independent standard normal r.v.’s, and  $\lambda_n > 0$, $n \in \mathbb{N}$, $\sum_n \lambda_n < \infty$ are the eigenvalues of the integral equation
\begin{equation}\label{int_eq}
\lambda f(x) = \int\limits_{\mathcal{O}} G_X(x, u)f(u)d\mu(u).
\end{equation}%
Thus we arrive at the equivalent problem of studying the asymptotic behavior as $\varepsilon\to 0$ of $\ln{\bf P} \{\sum_{n=1}^{\infty} \lambda_n\xi^2_n \leq \varepsilon^2 \}$. According to \cite{NazLog} the answer depends only on the main term of the asymptotics of the sequence $\lda_n$.%

The case of the purely power asymptotics $\lda_n \sim Cn^{-p}$, $p>1$, was considered in \cite{Zol1,Zol2,DudHof,Ibrag}. In \cite{NazNikKar} the case of regular asymptotics is considered, and in \cite{Naz} --- the case of almost power asymptotics with a periodic component.

Consider a more general case. Suppose
\begin{equation}\label{genldaasymp}
\lda_n(\mathcal{T}) = \phi(n) := \dfrac{\psi(n)\cdot \theta(\ln(n))}{n^p},
\end{equation}
where $p>1$, and functions $\theta$ is uniformly continuous on $\mathbb{R}$, bounded, separated from zero, and function $\phi(t)$ is monotonous on $\mathbb{R}$.
 
Function $\phi(n)$ satisfies the conditions of Theorem 2 from \cite{Lifsh1}, which for this case has the following form:

\begin{rusproposition}
\begin{equation}\label{asympP}
\mathbf{P}\bigg\{ \sum\limits_{n=1}^\infty \phi(n)\xi_n^2 \leq r \bigg\} \sim 
\dfrac{\exp (L(u)+ur)}{\sqrt{2\pi u^2 L''(u)}}, \quad r\to 0,
\end{equation}
where
$$
L(u) = \sum_{n=1}^\infty \ln f(u\phi(n)), \quad f(t) := (1+2t)^{-1/2},
$$ 
$u=u(r)$ is an arbitrary function satisfying
$$
\lim_{r\to 0} \dfrac{L'(u)+r}{\sqrt{L''(u)}} = 0.
$$
\end{rusproposition}

First, we analyse the asymptotics of $L'(u)$ as $u\to +\infty$. In our case 
$$
uL'(u) = -\sum_{n=1}^\infty \dfrac{u\psi(n)\theta(\ln(n))}{n^p+2u\psi(n)\theta(\ln(n))}\to -\infty, \quad u\to +\infty.
$$

Since $\phi(t)$ is a decreasing function, we can estimate
\begin{align*}
\sum_{n=1}^\infty \dfrac{u\psi(n)\theta(\ln(n))}{n^p+2u\psi(n)\theta(\ln(n))} &\geqslant
\int\limits_1^\infty\dfrac{u\psi(t)\theta(\ln(t))\,dt}{t^p+2u\psi(t)\theta(\ln(t))}
\\
&\geqslant \sum_{n=2}^\infty \dfrac{u\psi(n)\theta(\ln(n))}{n^p+2u\psi(n)\theta(\ln(n))}
\sim -uL'(u),
\end{align*}
thus
$$
uL'(u) \sim I_1(u) := -\int\limits_1^\infty\dfrac{u\psi(t)\theta(\ln(t))\,dt}{t^p+2u\psi(t)\theta(\ln(t))}.
$$
Replacing the integral interval with $(0, \infty)$ and substituting
$$
t =t(z) := z\phi^{-1}(1/u) = z\gamma(u),
$$
$$
\gamma(u) := \phi^{-1}(1/u) \sim u^{1/p}\varphi(u)\vartheta(\ln(u)), \quad u\to \infty,
$$
where $\varphi$ is a SVF, and $\vartheta$ is uniformly continuous, bounded, separated from zero function, we obtain
$$
I_1(u)=-\gamma(u)\cdot \int\limits_0^\infty \dfrac{dz}{2+z^p\cdot \dfrac{(\gamma(u))^p}{u\psi(t(z))\theta(\ln(t(z)))}} + O(1), \quad u\to\infty.
$$
From $1/u=\phi(\gamma(u))$ we obtain the relation
$$
(\gamma(u))^p/u = \psi(t(z)/z)\theta(\ln(t(z)/z)).
$$
Substituting it into the integral and considering the definition of $\gamma(u)$, we obtain
$$
I_1(u)=-\gamma(u)\cdot \int\limits_0^\infty \dfrac{dz}{2+z^p\cdot \dfrac{\psi(\gamma(u))\theta(\ln(\gamma(u)))}{\psi(z\gamma(u))\theta(\ln(z\gamma(u)))}} + O(1), \quad u\to\infty.
$$
It is clear, that
$$
\theta(\ln( z\gamma(u) )) =  \theta\Big(\frac{\ln(u)}{p}+\ln(z)\Big)(1+o(1)), \quad u\to\infty.
$$
Note also, that according to Proposition~\ref{propSVF1}, part 2, for every $\varepsilon>0$ the ratio $\psi(t)/t^\varepsilon$ decreases at large values of
$t$, thus, for $z>1$ 
$$
\dfrac{\psi(t)}{\psi(zt)} = \dfrac{1}{z^\varepsilon}\cdot \dfrac{\psi(t)}{t^\varepsilon} \cdot \dfrac{(zt)^\varepsilon}{\psi(zt)} \geqslant \dfrac{C(\varepsilon)}{z^\varepsilon}.
$$
This gives us a majorant to use the Lebesgue theorem. As a result, we have
\begin{equation}\label{I1asymp}
I_1(u)=-u^{1/p}\vartheta(u)\cdot \int\limits_0^\infty \dfrac{dz}{2+z^p\cdot \dfrac{\theta(\ln(u)/p)}{\theta(\ln(u)/p+\ln(z))}} + O(1), \quad u\to\infty.
\end{equation}
Since the integral is a uniformly continuous, bounded and separated from zero function of $\ln(u)$, we obtain
\begin{equation}\label{Theta1asymp}
L'(u)\sim -u^{-\frac{p-1}{p}}\varphi(u)\vartheta_1(\ln(u)), \quad u\to\infty,
\end{equation}
where $\varphi$ is a SVF from the asymptotics of $\gamma$, and $\vartheta_1$ is a uniformly continuous, bounded and separated from zero function.

Similarly we obtain
\begin{equation}\label{Lderder}
u^2 L''(u) \sim 2\int\limits_1^\infty\dfrac{(u\psi(t)\theta(\ln(t)))^2\,dt}{(t^p+2u\psi(t)\theta(\ln(t)))^2} \asymp u^{1/p}\varphi(u),
\end{equation}
$$
L(u)\sim -\dfrac{1}{2}u^{1/p}\varphi(u)\vartheta(\ln(u))\cdot \int\limits_0^\infty 
\ln \left( 1+ \dfrac{2\theta(\ln(u)/p+\ln(z))}{z^p \theta(\ln(u)/p)} \right)\,dt
.
$$

\noindent Since $L''(u)>0$, the equation $L'(u)+r=0$ has for sufficiently small $r$ a unique solution $u(r)$, such that $u(r)\to\infty$ as $r\to 0$. Moreover, the relation \eqref{Theta1asymp} gives us
\begin{equation}\label{asympu}
u(r)\sim r^{-\frac{p}{p-1}} \eta(1/r)\vartheta_2(\ln(1/r)), \quad r\to 0,
\end{equation}
where $\eta$ is a SVF, and $\vartheta_2$ is a uniformly continuous, bounded and separated from zero function.

\noindent Substituting \eqref{Lderder} into \eqref{asympP}, we conclude, that
\begin{align}
\begin{split}\label{Logasymp}
\ln \mathbf{P}&\bigg\{ \sum\limits_{n=1}^\infty \phi(n)\xi_n^2 \leq r \bigg\} \sim
L(u)+ur = L(u) -uL'(u)
\\
&\sim -u^{1/p}\varphi(u)\vartheta(\ln(u))\cdot 
\int\limits_0^\infty \Bigg[ 
\frac{1}{2}\ln \bigg( 1+ \dfrac{2\theta(\ln(u)/p+\ln(z))}{z^p \theta(\ln(u)/p)} \bigg)
\\
&\qquad\qquad\qquad\qquad\qquad -\dfrac{1}{2+z^p\cdot \dfrac{\theta(\ln(u)/p)}{\theta(\ln(u)/p+\ln(z))}}
\Bigg] \, dz
.
\end{split}
\end{align}

\noindent What's left is to note, that the integrand
$$
\dfrac{1}{2}\ln(1+2x)- \dfrac{x}{2x+1} 
$$
is positive, thus the integral is a uniformly continuous function of $\ln(u)$, bounded and separated from zero.
Substituting the asymptotics of $u$ obtained above and replacing $r$ with $\varepsilon^2$, we formulate the following theorem.

\begin{rustheorem}\label{smalldev}
Let the eigenvalues of \eqref{int_eq} have the form \eqref{genldaasymp}. Then, as $\varepsilon\to 0$,
\begin{equation}\label{LogAsymp1}
\ln \mathbf{P}\left\{ \|X\|_\mu \leq \varepsilon \right\} \sim
-\varepsilon^{-\frac{2}{p-1}}\xi(1/\varepsilon)\zeta(\ln(1/\varepsilon)),
\end{equation}
where $\xi$ is a SVF, $\zeta$ is a uniformly continuous, bounded and separated from zero function. Moreover, if the function $\theta$ in \eqref{genldaasymp} is asymptotically $\frac{T}{p}$-periodic, then the function $\zeta$ might be chosen to be  $\frac{T(p-1)}{2p}$-periodic.
\end{rustheorem}

\begin{proof}
The first statement of the theorem follows from \eqref{Logasymp} and \eqref{asympu}, if we replace $r$ with $\varepsilon^2$. Further, if $\theta$ is asymptotically  $\frac{T}{p}$-periodic, then the function~$\vartheta$ is asymptotically $T$-periodic, and by the Lebesgue theorem it is easy to confirm, that the integrals in \eqref{I1asymp} and 
\eqref{Logasymp} are also asymptotically $T$-periodic functions of $\ln(u)$. Thus, the function $\vartheta_1$ in \eqref{Theta1asymp} is asymptotically $T$-periodic, hence  $\vartheta_2$ in \eqref{asympu} is asymptotically $\frac{T(p-1)}{p}$-periodic.
What's left is to note, that it follows from \eqref{asympu}, that 
$$
\ln(u) = \frac{p}{p-1}\ln(1/r) (1+o(1)),\quad r\to 0,
$$ 
thus the integral in \eqref{Logasymp} and the function $\vartheta(\ln(u))$ are asymptotically $\frac{T(p-1)}{p}$-periodic functions of $\ln(1/r)$, and the second statement is proven.
\end{proof}

\noindent Now, let there be two Gaussian processes $X(x)$, $x \in \mathcal{O}_1 \subseteq \mathbb{R}^{m_1}$, and $Y(y)$, $y \in \mathcal{O}_2\subseteq \mathbb{R}^{m_2}$, with zero mean and covariation functions $G_X(x, u)$, $x, u \in\mathcal{O}_1$, and $G_Y(y, v)$, $y, v \in \mathcal{O}_2$, correspondingly. Consider a new Gaussian function $Z(x, y)$, $x \in \mathcal{O}_1$, $y \in \mathcal{O}_2$, with zero mean and the covariation $G_Z((x, y),(u, v)) = G_X(x, u)G_Y(y, v)$. Such a Gaussian function obviously exists, and the integral operator with the kernel $G_Z$ is the tensor product of the operators with the kernels $G_X$ and $G_Y$. Therefore, we use the notation $Z = X \otimes Y$ and we call the process $Z$ the tensor product of processes $X$ and $Y$. The generalization to the multivariate case when obtaining $\bigotimes_{j=1}^d X_j$ is straightforward.

\begin{rusexample}
Let us demonstrate the application of the Theorems from \S 4 for example for the Brownian sheet 
$$
\mathbb{W}_d(x_1, \dots, x_d) = W_1(x_1) \otimes W_2(x_2) \otimes \dots \otimes
W_d(x_d)
$$
in the unit cube with the norm $L_2(\mu)$, where $\mu = \bigotimes\limits_{j=1}^d \mu_j$, and every measure $\mu_j$ is a selfsimilar measure of a generalized Cantor type. Spectral asymptotics of the operators-multipliers in this case are known from \cite{KL} and \cite{SV}:
$$
\mathcal{N}_j(t) \sim \dfrac{s_j(\ln(1/t))}{t^{1/p_j}}, \quad t\to 0+,
$$
where $s_j$ are continuous and $T_j$-periodic, $p_j > 1$. This power asymptotics were considered in the Example~\ref{LogExample} and correspond to the case $\varkappa_1=\varkappa_2=0$.

For certain measures $\mu_j$ functions $s_j$ could be constant, but \cite{VladSheip,Rast} describe wide classes of measures, for which the inconstancy of the periodic component is proven.

Let $\mathfrak{p} := p_1 = \min p_j$. First, we use Theorem \ref{Th1} for each operator with $p_j > \mathfrak{p}$, multiplying it with the first one. As a result, we can assume, without loss of generality, that all of the operators asymptotics have the same power exponent.

If among the rest of the operators at least one has a degenerated periodic component, then the tensor product will also have a degenerated periodic component. If at least two periods are incommensurable, then the periodic component of their tensor product will degenerate into constant according to the Example \ref{LogExample}, and as a result, the periodic component of the whole tensor product will also degenerate into constant.

If all power exponents coincide and all periods are commensurable, then by using the Example \ref{LogExample} we obtain
$$
\mathcal{N}_\otimes(t) \sim \dfrac{C\ln^{\mathfrak{d}-1}(1/t)s^{(\mathfrak{d})}_\otimes(\ln(1/t))}{t^{1/\mathfrak{p}}}, \quad t\to 0+,
$$ 
where $\mathfrak{d}$ is the number of the power exponents equal to $\mathfrak{p}$, $s^{(\mathfrak{d})}_\otimes$ is obtained by iterating the formula \eqref{sotimesdef} required number of times. This allows us to use for this Gaussian field Theorem \ref{smalldev}, Moreover, by direct calculation we discover, that in \eqref{asympu} and \eqref{LogAsymp1}
$$
\eta(1/r)\sim \ln^{\frac{(\mathfrak{d}-1)\mathfrak{p}}{\mathfrak{p}-1}}(1/r), \quad r\to 0,
$$
$$
\xi(1/r)\sim \ln^{\frac{(\mathfrak{d}-1)\mathfrak{p}}{\mathfrak{p}-1}}(1/r), \quad r\to 0.
$$
Thus, as $\varepsilon\to 0$, we have
$$
\ln \mathbf{P}\left\{ \|\mathbb{W}_d\|_\mu \leq \varepsilon \right\} \sim
-\varepsilon^{-\frac{2}{\mathfrak{p}-1}}\ln^{\frac{(\mathfrak{d}-1)\mathfrak{p}}{\mathfrak{p}-1}}(1/\varepsilon)\zeta(\ln(1/\varepsilon)),
$$
where $\zeta$ is a certain $\frac{T(\mathfrak{p}-1)}{2\mathfrak{p}}$-periodic function.

Consider the simplest case, when all measures are classical Cantor measures. For this case we know the values
$$
p = \log_2 6, \quad T = \ln 6.
$$
Substituting this values into the asymptotics, as $\varepsilon\to 0$, we obtain
$$
\ln \mathbf{P}\left\{ \|\mathbb{W}_d\|_\mu \leq \varepsilon \right\} \sim
-\varepsilon^{-2\log_3 2}\ln^{(d-1)\log_3 6}(1/\varepsilon)\zeta(\ln(1/\varepsilon)),
$$
where $\zeta$ is a certain $\frac{\ln 3}{2}$-periodic function.
\end{rusexample}
\begin{rusremark}
Similar results hold, if instead of Wiener process we consider different independent Green Gaussian processes. Some examples of well-known Green Gaussian processes could be found in \cite{Naz}.
\end{rusremark}

\subsection*{Acknowledgements}

Author is grateful to A.~I.~Nazarov for the problem statement and attention, and to D.~D.~Cherkashin for valuable remarks. 

\medskip
The main results of this paper (Theorems 1--7) were obtained with the support of RSF, grant 14-21-00035. Results of \S 5 (Theorem 8) were obtained with the support of RFBR grant (project 16-01-00258a).
\bigskip


\begin{rusbibliography}{99}

\bibitem{GrLuPa} 
Graf S., Luschgy H., Pag\`{e}s G., 
\emph{Functional quantization and small ball probabilities for Gaussian processes}, 
J.  Theoret. Probab. {\bf 16} (2003), no.~4, 1047--1062. 

\bibitem{LuPa} 
Luschgy H., Pag\`{e}s G., 
\emph{Sharp asymptotics of the functional quantization problem for Gaussian processes}, 
Ann. Probab. {\bf 32} (2004),  no.~2, 1574--1599.

\bibitem{PaWa} 
Papageorgiou A.,  Wasilkowski G.~W., 
\emph{On the average complexity of multivariate problems}, 
J. Complexity  {\bf 6} (1990), no.~1,  1--23.

\bibitem{NazNikKar}  
Karol' A.,  Nazarov A.,  Nikitin Ya., 
\emph{Small ball probabilities for Gaussian random fields
and tensor products of compact operators}, Trans. 
Amer. Math. Soc. {\bf 360} (2008), no.~3, 1443--1474.

\bibitem{NazKar} 
Karol' A.,  Nazarov A., 
\emph{Small ball probabilities for smooth Gaussian fields and tensor products of 
compact operators}, Math. Nachr. {\bf 287} (2014), no.~5--6,  595--609.

\bibitem{KL} 
Kigami J.,  Lapidus M.~L., 
\emph{Weyl`s problem for the spectral 
distributions of Laplacians on p.c.f. self-similar fractals}, 
Comm. Math. Phys. {\bf 158} (1991), no.~1, 93--125.

\bibitem{SV} 
Solomyak M., Verbitsky E., 
\emph{On a spectral problem related to self-similar measures}, 
Bull. London Math. Soc.  {\bf 27} (1995), no.~3, 242--248.

\bibitem{Naz}
Nazarov A.~I., 
\emph{Logarithmic $L_2$-small ball asymptotics with respect to self-similar measure
for some Gaussian processes}, 
Zap. Nauchn. Semin. St.-Petersb. Otdel. Mat. Inst.
Steklov (POMI), {\bf 311} (2004), 190--213 (in Russian). 
J. Math. Sci. (N. Y.), 133:3 (2006), 1314--1327.

\bibitem{Sytaya}
Sytaya G.~N., 
\emph{On some asymptotic representations of the Gaussian measure in a Hilbert
space}, 
Theory of Stochastic Processes, Kiev, {\bf 2} (1974), 93--104 (in Russian).  

\bibitem{Lifsh} 
Lifshits M.~A., 
\emph{Asymptotic behavior of small ball probabilities}, 
Prob. Theory and Math. Stat.,  VSP/TEV,  Vilnius, 1999, pp.~ 453--468.

\bibitem{LiShao} 
Li W.~V.,  Shao Q.~M., 
\emph{Gaussian processes: inequalities, small ball probabilities and
applications}, Stochastic Processes: Theory and Methods, Handbook of Statistics,
vol.~19, North--Holland, Amsterdam, 2001, pp.~533--597.

\bibitem{Site} 
Small deviations for stochastic processes and related topics, Internet site,
http://www.proba.jussieu.fr/pageperso/smalldev/

\bibitem{Csaki}  
Cs\`{a}ki E., 
\emph{On small values of the square integral of a multiparameter Wiener process},
Statistics and Probability (Visegraid, 1982), Reidel, Dordrecht, 1984, pp.~19--26.

\bibitem{Li} 
Li W.~V., 
\emph{Comparison results for the lower tail of Gaussian seminorms}, 
J. Theor. Probab. {\bf 5} (1992), no.~1,  1--31. 

\bibitem{Seneta} 
Seneta E., 
\emph{Regularly varying functions}, Lecture Notes in Math.,
vol.~ 508, Springer-Verlag, Berlin, 1976.

\bibitem{VladSheip}
Vladimirov A.~A., Sheipak I.~A.,
\emph{On the Neumann Problem for the Sturm–Liouville Equation with Cantor-Type Self-Similar Weight},
Functional Analysis and Its Applications. {\bf 47} (2013), no.~4, 261--270.

\bibitem{Vlad}
Vladimirov A.~A., 
\emph{Method of oscillation and spectral problem for four-order differential operator with self-similar weight}
Algebra i Analiz, {\bf 27} (2015), no.~2, 83--95 (in Russian). 
St. Petersburg Math. J., {\bf 27} (2016), no.~2, 237--244.

\bibitem{Rast}
Rastegaev N.~V., 
\emph{On spectral asymptotics of the Neumann problem for the Sturm–Liouville equation with self-similar generalized Cantor type weight},
Zap. Nauchn. Semin. St.-Petersb. Otdel. Mat. Inst.
Steklov (POMI), {\bf 425} (2014), 86--98 (in Russian).
J. Math. Sci. (N. Y.), {\bf 210} (2015), no.~6, 814--821.

\bibitem{DamanikGorodetskiSolomyak} 
Damanik D.,  Gorodetski A.,  Solomyak B., 
\emph{Absolutely continuous convolutions of singular measures and an application to the square 
Fibonacci Hamiltonian}, Duke Math. J. {\bf 164} (2015), no.~8, 1603--1640.

\bibitem{Oxtoby} 
Oxtoby J.~C.,
\emph{Ergodic sets}, Bull. Amer. Math. Soc. {\bf 58} (1952), no.~2, 116--136.


\bibitem{NazLog}  
Nazarov A., 
\emph{Log-level comparison principle for small ball probabilities}, 
Statist. Probab. Lett. {\bf 79} (2009), no.~4, 481--486. 

\bibitem{Zol1} 
Zolotarev V.~M., 
\emph{Gaussian measure asymptotic in $l_2$ on a set of centered spheres with radii tending to zero},
Proc. 12th Europ.  Meeting of Statisticians, Varna, 1979, pp.254. 

\bibitem{Zol2}
Zolotarev V.~M., 
\emph{Asymptotic behavior of Gaussian measure in l2, Problems of stability
of stochastic models}, 
Proc. Semin., Moscow, 1984, 54–-58 (in Russian). 
J. Sov. Math., 24 (1986), 2330–-2334.

\bibitem{DudHof} 
Dudley R.~M.,  Hoffmann-J{\o}rgensen J.,  Shepp L.~A.,
\emph{On the lower tail of Gaussian seminorms}, Ann. Prob. {\bf 7} (1979), no.~2, 319--342. 

\bibitem{Ibrag}
Ibragimov I.~A., 
\emph{The probability of a Gaussian vector with values in a Hilbert space
hitting a ball of small radius}, 
Zap. Nauchn. Semin. Leningrad. Otdel. Mat. Inst. Steklov
(LOMI), {\bf 85} (1979), 75–-93 (in Russian). 
J. Sov. Math., {\bf 20} (1982), 2164-–2174.

\bibitem{Lifsh1} 
Lifshits M.~A., 
\emph{On the lower tail probabilities of some random series}, 
Ann. Prob. {\bf 25}  (1997), no.~1, 424--442.

\end{rusbibliography}
\end{document}